\documentclass[11pt,article]{amsart}
\usepackage{amsmath,amssymb,epsfig,color, amsthm, mathrsfs}
\usepackage{wrapfig,lipsum,floatrow,sidecap,comment}
\usepackage{graphicx}
\usepackage{multirow}
\usepackage{hhline}
\usepackage{url}
\usepackage{soul}
\usepackage{ulem}

\graphicspath{ {images/} }
\allowdisplaybreaks
\newtheorem{theorem}{Theorem}[section]
\newtheorem{lemma}[theorem]{Lemma}
\newtheorem{proposition}[theorem]{Proposition}

\theoremstyle{definition}
\newtheorem{definition}[theorem]{Definition}
\newtheorem{example}[theorem]{Example}

\theoremstyle{remark}
\newtheorem{remark}[theorem]{Remark}

\numberwithin{equation}{section}

\def\sideremark#1{\ifvmode\leavevmode\fi\vadjust{\vbox
to0pt{\vss \hbox to 0pt{\hskip\hsize\hskip1em
\vbox{\hsize2cm\tiny\raggedright\pretolerance10000
\noindent#1\hfill}\hss}\vbox to8pt{\vfil}\vss}}}

\begin{document}
\title[ Dilations for operator-valued quantum measures]{ Dilations for operator-valued quantum measures}
\author{Deguang Han}
\address{Department of Mathematics, University of Central
Florida, Orlando, USA} \email{deguang.han@ucf.edu}
\author{Qianfeng Hu}
\address{School of Mathematical Sciences and LPMC, Nankai University, Tianjin, China} \email{qianfenghu@mail.nankai.edu.cn}
\author{David R. Larson}
\address{Department of Mathematics, Texas A\&M University, College Station, USA}
\email{larson@math.tamu.edu}
\author{Rui Liu}
\address{School of Mathematical Scicens and LPMC, Nankai University, Tianjin, China}
\email{ruiliu@nankai.edu.cn}
\begin{abstract} This paper concerns the dilations of Banach space operator-valued quantum measures. While the recently developed general dilation theory can lead to a projection (idempotent) valued dilation for any quantum measure over the projection lattice for a von Neumann algebra that dose not contain type $I_{2}$ direct summand, such a dilation does not necessarily guarantee  the preservation of countable additivity of the quantum measure.  So it remain an open question whether every countably additive $B(X)$-valued quantum measure can be dilated to a countably additive projection-valued measure.The main purpose of  this paper  is to prove that such a dilation can be constructed if  one of the following two conditions is satisfied: (i)  the underling Banach space $X = \ell_{p}$  $(1\leq p <  2$)  or  it has Schur property,  (ii) the quantum measure  has bounded $p$-variation  for some $ 1\leq p < \infty $. All of these were achieved by establishing a non-commutative version of a minimal dilation theory on the so-called elementary dilation space equipping with  an appropriate dilation norm. In particular,  the newly introduced $p$-variation norm on the elementary dilation space allows us to prove that  every operator-valued quantum measure with bounded $p$-variation has a projection-valued quantum measure dilation that preserves the boundedness of the $p$-variation.
\end{abstract}
\thanks{Deguang Han acknowledges the support from NSF under the grant DMS-1712602. Qianfeng Hu and Rui Liu acknowledge partial support by NSFC grants 11671214, 11971348 and 12071230, and Hundred Young Academic Leaders Program of Nankai University.}
\date{}


\maketitle

\section{Introduction}
A  recently developed general dilation theory  for operator-valued measures (respectively,  for bounded linear maps) tells us that every operator-valued measure can be dilated to a projection valued measure, and every norm-continuous linear map on a Banach algebra can be dilated to norm-continuous homomorphism  acting on a Banach space \cite{HLLL1}. However, it is generally unknown whether some other types of continuity can also be preserved via dilations. One of the most interesting cases is the preservation of the ultraweak continuity of the linear maps. In terms of  operator-valued measures, it  means preserving the countable additivity of the dilated measures.  In the commutative domain case, this question has been settled for  purely atomic measures \cite{HLLL1} and for general operator-valued measures \cite{HLL}. This paper continues this investigation  by focusing on quantum measures where the domain is the projection lattice of an arbitrary von Neumann algebra (VN algebra for short).

The quantum measure theory \cite{Gl,Ha1} has its origin in mathematical formalism of quantum mechanics and is often viewed as ``noncommutative" or ``quantum" analogs of classical measure theory.  In classical measure theory, the basic concept is a measure on the $\sigma$-algebra of subsets. While in  quantum measure theory, the measure is acting on the projection lattice of a VN algebra, or equivalently, on the lattice structure of closed subspaces in a Hilbert space. Intuitively, the transition from the $\sigma$-algebra to projection lattice includes replacing the disjointedness of subsets by orthogonality of subspaces. But from abstract viewpoint,while a $\sigma$-algebra of subsets is a Boolean algebra with respect to set theoretic operations, a projection lattice is not  necessarily Boolean.  A  typical example is  the lattice of projections in $B(H)$  which is never distributive (unless $H$ is one dimensional). Consequently,  it usually requires a substantially different  set of techniques in order to establish a non-commutative dilation theory for quantum measures.

Obviously, every bounded linear map from a VN algebra $M$ to a Banach space $X$ induces a  vector-valued quantum measure on the projection lattice $\mathcal{P}(M)$ of $M$.  However, the converse is not necessarily true since there exist counterexamples of scalar-valued measures that fail to extend to a linear functional when $M$ is the algebra of $2 \times 2 $ matrices.  Recall that a  VN algebra $M$ is said to be $\sigma$-finite if every family of non-zero pairwise orthogonal projections in $M$ is at most countable \cite{Ta}. In this paper we will be mainly focused  on the  $\sigma$-finite VN algebras without  type $I_{2}$ direct summand.  The celebrated Gleason theorem \cite{Gl} asserts that every bounded completely additive measure on the projection lattices $\mathcal{P}(H)$, where $H$ is a Hilbert space with $\dim H$ $\geq 3$,  extends uniquely to a normal functional on the algebra $B(H)$. Bunce and Wright's generalization  to vector measures \cite{BW} states that if a VN algebra  $M$ has no summand of type $I_{2}$, then every bounded finitely additive measure from $\mathcal{P}(M)$ to Banach space $X$ extends uniquely to a bounded linear operator from $M$ to $X$. In particular, for a Banach space $X$,  every finitely additive $B(X)$-valued measure extends to a bounded linear map. By the dilation theorem \cite{HLLL1} for bounded linear maps,  this linear map has bounded homomorphism dilation. Therefore, every finitely additive operator-valued measure can be dilated to a projection-valued measure.

Since ultraweak continuity (resp. countable additivity) often occurs when dealing with linear maps (resp. measures) on VN algebras, it is nature to ask whether there exists a dilation that preserves such a continuity (resp. additivity). While this question has been answered affirmatively for abelian VN algebras \cite{HLLL1, HLL},  we still don't know if this remains true for arbitrary VN algebras. In this paper   we  are able to prove  that this can be achieved when imposing additional assumptions either on $X$ or on the quantum measure. More precisely, our first main result states that  if  $X$ is either  $\ell_{p} (1\leq p < 2)$ or  has Schur property, then every ultraweakly-wot continuous bounded linear  mapping can be dilated  to a ultraweakly-wot continuous Jordan homomorphism. Our second main result involves the concept of  operator-valued quantum measures with bounded $p$-variation. This  concept was introduced in \cite{HLL}   for commutative cases, and it was proved that such a measure can be dilated to a projection-valued measure while preserving such a property. We  will generalize this concept to operator-valued measures on the projection lattice of arbitrary VN algebras, and prove that every such a measure can be dilated to a projection-valued measure with bounded total $p$-variation. Since such a measure is always countably additive, this dilation preserves the countable additivity. In summary, the following is one of the main results of this paper.






\begin{theorem} \label{p-dilation} Let $M$ be a $\sigma$-finite VN algebra without type $I_{2}$ direct summand and $U: \mathcal{P}(M) \to B(X)$ be a countably additive quantum  measure, where $X$ is a Banach space. Then $U$ can be dilated to  a  countably additive projection-valued quantum  measure $V: \mathcal{P}(M)\to B(Z)$ if one of the following conditions is satisfied:

(i) $X = \ell_{p}\, (1\leq p < 2)$) or  has Schur property.

(ii)  $U$ has bounded $p$-variation for some $1\leq p < \infty$.
\end{theorem}


\noindent{\bf Outline of the paper:}   We will recall or introduce in section 2 some  necessary definitions and notations about quantum measures.   The general dilation theory for operator-valued measures over the  classical commutative domain case was mainly based on an elementary dilation theory developed in \cite{HLL, HLLL1, HLLL2}.
In order to prove our main results, we need  an analogous of such a theory for quantum measures.  So, in section 3,  we  will first  introduce an elementary dilation space. While the projection-valued dilation of a quantum measure is not unique, we prove that every such a dilation  naturally induces a projection-valued dilation on the elementary space. The proof of the first part of Theorem \ref{p-dilation} is then achieved by  introducing an appropriate dilation norm on this elementary space. The main results  of this section are presented in Theorem \ref{thm3.1}, Proposition \ref{dilationnorm}, Theorem \ref{countablemeasuredilation} and Theorem \ref{maptojordan}.
Section 4 is devoted to establishing a dilation theory for quantum measures with bounded $p$-variations.  We first introduce the concept of $p$-variation for quantum measures, and  show that the definition of bounded $p$-variation  coincides with the one defined for operator valued measures in the commutative domain (i.e., abelian VN algebra) case. Then we introduce a $p$-variation norm on the elementary dilation space, and prove that every quantum measure with bounded $p$-variation can be dilated to a same type of projection-valued measure on the elementary dilation space equipped with the $p$-variation norm. This proves the second part of Theorem \ref{p-dilation}.

\section{Preliminaries}


Let $X$ be a Banach space, and $M$ be a VN algebra. We will use the following list of notations throughout the rest of this paper.

\begin{itemize}

\item $B(X)$ -- the space of all bounded linear operators on $X$.
\item  $B_{M}$ --  the closed unit ball of $M$; $M_{sa}$ --  the set of self-adjoint operators in $M$;   $M_{+}$ -- the set of all positive operators in $M$, $I$ -- the identity operator.

\item $\mathcal{P}(M)$ denotes the set of all projections in $M$,  that is the set of all self-adjoint idempotents. $\mathcal{P}(M)$ is ordered by order relation $P \leq Q$ if $PQ=P.$

\item Supremum and infimum of two projections $P$ and $Q$ will be denoted by $P \vee Q$ and $P \wedge Q,$ respectively. We use $P \perp Q $  to denote that  projections $P$ and $Q$ are orthogonal.

\item $M_{*}$ is  the predual space of $M$, which is the Banach space of all ultraweakly continuous (\textit{normal}) linear functionals on $M$.

\end{itemize}

For a mutually orthogonal family of  projections $\{P_{\alpha}\}_{\alpha\in \mathbb{I}}$,  we will also use $\sum_{\alpha\in \mathbb{I}}P_{\alpha}$ to denote  $\vee_{\alpha\in \mathbb{I}}P_{\alpha}$ ($\mathbb{I}$ is the  index set and the convergence is in strong operator topology).  For operators $\{ A_{i}\}_{i\in \mathbb{I}}\subset  M,$  we use the notation  $ A_{i} \stackrel{sot}{\longrightarrow}{A}$ ($ A_{i} \stackrel{wot}{\longrightarrow}{A}$) for the strong operator topology (weak operator topology, respectively) convergence.  As ultraweak topology and weak* topology coincide on $M$, we will use $ A_{i} \stackrel{w*}{\longrightarrow}{A}$ to denote that $A_{i}$ converges to $A$ in ultraweak topology of $M$.  We refer to \cite{Ta, Ha2, PX, Ha1} for more backgrounds about operator algebras and noncommutative quantum measure theory, and \cite{HLLL1, HLL, Pa}  about the dilation theory of operator-valued measures. Analogous to \cite{BW}, we will use the following definition for operator-valued quantum measure.

\begin{definition}\label{quantummeasure}
Let $\mathcal{P}(M)$ be the lattice of projections of a VN algebra $M$ and $X$ be a Banach space.  If
$$
U : \mathcal{P}(M)\to B(X)
$$
is a mapping such that
$$
U(P_{1}+P_{2})=U(P_{1})+U(P_{2}), \text{ whenever }  P_{1}\perp P_{2},
$$
then $U$ is said to be  a (finitely additive) $B(X)$-valued measure on $\mathcal{P}(M)$. We say that $U$ is

\begin{itemize}

\item[(i)] \textit{countably additive (or $\sigma$-additive)} if  $U(\sum_{n}P_{n})=\sum_{n}U(P_{n})$ for  any countable family $P_{n}$ of mutually orthogonal projections and  converges unconditionally with respect to weak operator topology ($wot$) of $B(X).$

\item[(ii)] \textit{bounded} if $\sup\{\|U(P)\|: P \in \mathcal{P}(M)\}<\infty $ and the supremum is called the norm of $U$ and denoted as $\|U\|$.

\item[(iii)] \textit{projection-valued} if  $U(P)$ is a projection (i.e. idempotent) on $X$ for every $P \in \mathcal{P}(M) $.

\item[(iv)] \textit{self-adjoint} if $X$ is a Hilbert space $\mathcal{H}$ and $U(P)^{*}=U(P) $ for all $P \in \mathcal{P}(M) $ and \textit{positive} if $ U(P) \in B(\mathcal{H})_{+}$ for all $P \in \mathcal{P}(M) $.
\end{itemize}

\end{definition}

%


\begin{remark} \label{Remarkondef} We have a few comments about the above definition.

\begin{itemize}

\item[(i)]  When dealing with the operator-valued measure $U$ on a $\sigma$-finite VN algebra, countable additivity of $U$ is equivalent to being completely additive \cite{Ta}.  Moreover  completely additive or countably additive bounded linear maps from $M$ to $B(X)$ are defined in a similar way.

\item[(ii)] The boundedness of the countably additive operator-valued measure on the classical measure space is an essential property \cite{HLLL1} in the study of operator-valued measures. However,  the quantum measures on the projection lattice are not necessarily bounded. We mention that Dorofeev \cite[Theorem 2]{Do} establishes some boundedness theorems for quantum measures. In what follows we  will always assume that a quantum measure is finitely additive and bounded. If countable additivity or complete additivity is assumed, it will be explicitly stated.

\item[(iii)]  If we assume $U$ is a countably additive quantum measure, that is,  for any  countable collection of mutually orthogonal projections $\{P_{n}\}_{n=1}^{\infty}$ of $\mathcal{P}(M)$ with supremum $P$, then $x^{*}(\sum_{n=1}^{\infty}U(P_{n})x)$ converges to $x^{*}(U(P)x)$  for all $x \in X $ and $x^{*} \in X^{*}$,  i.e.,  $\sum_{n=1}^{\infty}U(P_{n})x $ weakly unconditional converges to $ U(P)x$. The Orlicz-Pettis theorem states that weakly subseries convergence and norm subseries convergence of a series are the same in every Banach space.  Furthermore, since $\mathcal{P}(M)$ is a complete lattice, we obtain that  weakly countably additive vector measures are countably additive. Thus, the Definition \ref{quantummeasure} for countably additive quantum measure is equivalent to say  that $ U $ is \textit{strong countably additive} on $X$, that is,
$$
\sum_{n=1}^{\infty}U(P_{n})x= U(P)x ,\quad \forall x\in X.
$$
\end{itemize}
\end{remark}


The ultraweak-topology on $B(\mathcal{H})$ for a Hilbert space $\mathcal{H}$ is well-understood. For a Banach space $X$, we can define the ultraweak-topology on $B(X)$ through the natural embedding $B(X) \hookrightarrow B\left(X, X^{* *}\right)$ and tensor products: Let $X \otimes Y$ be the tensor product of the Banach space $X$ and $Y$. The projective norm on $X \otimes Y$ is defined by:
$$
\|u\|_{\wedge}=\inf \left\{\sum_{i=1}^{n}\left\|x_{i}\right\|\left\|y_{i}\right\|: u=\sum_{i=1}^{n} x_{i} \otimes y_{i}\right\}
$$
We will use $X \otimes_{\wedge} Y$ to denote the tensor product $X \otimes Y$ endowed with the projective norm $\|\cdot\|_{\wedge} .$ Its completion will be denoted by $X \widehat{\otimes} Y$. For any Banach spaces $X$ and $Y$, we have the identification:
$$
(X \widehat{\otimes} Y)^{*}=B\left(X, Y^{*}\right)
$$
Thus $B\left(X, X^{* *}\right)=\left(X \widehat{\otimes} X^{*}\right)^{*}$. Viewing $X \subset X^{* *},$ we define the ultraweak topology on $B(X)$ to be the weak* topology induced by the predual $X \widehat{\otimes} X^{*}$ by the following:
$T_{\alpha} \rightarrow T$ in the ultraweak topology if $F\left(T_{\alpha}\right) \rightarrow F(T)$ for any $F(A)=\sum_{n=1}^{\infty} f_{n}\left(A x_{n}\right)$
with $\sum_{n=1}^{\infty}\left\|f_{n}\right\|\left\|x_{n}\right\|<\infty .$ From the definition it is obvious that if $\sup \left\{\left\|T_{i}\right\|: i \in I\right\}$
is bounded, then $\left\{T_{i}\right\}$ is ultraweakly convergent to $T$ if and only if it is convergent to $T$ in the weak operator topology. We will usually use the term \textit{normal} to denote an ultraweakly continuous linear map.
\vspace{2mm}

Notation: $J$ is the canonical embedding from $B(X)$ into $B\left(X, X^{**}\right)=(X \widehat{\otimes} X^*)^*$.
The following lemma  is standard and will  be used in  the proof of  Proposition \ref{continuity}.

\begin{lemma}\cite{Ta}\label{mormal}
If $\rho$ is a bounded linear functional on a VN algebra $M$, then $\rho$ is normal if and only if $\rho$ is completely additive.
\end{lemma}


\begin{proposition}\label{continuity}
Let $\Phi$ be a bounded linear map from a VN algebra $M$ to $B(X)$.  Then the following  are equivalent.
\begin{enumerate}
	\item [(1)] $\Phi$ is completely additive.
	\item [(2)] $\Phi$ is ultraweakly-wot continuous.
	\item [(3)] $J \Phi$ is normal.
\end{enumerate}
\end{proposition}

\begin{proof}
The equivalence $(1)\Leftrightarrow (2)$ is clear by Lemma \ref{mormal}, for any $x \in X$ and $x^*\in X^*$, the bounded linear functional $x^*(\Phi(\cdot)x)$ on $M$ is completely additive if and only if $x^*(\Phi(\cdot)x)$ is normal.

$(2)\Rightarrow (3)$ It is sufficient to prove that, for any $ u=\sum_{k=1}^{\infty}x_{k}\otimes x_{k}^{*}\in X \widehat{\otimes} X^{*}$ with $\sum_{k=1}^{\infty}\|x_{k}\|\|x_{k}^{*}\|< +\infty $, $(J\Phi(\cdot))(u)$ is a normal functional on $M$.
Let $\phi_{k}=x_{k}^{*}(\Phi(\cdot)x_{k})$ for each $k\in\mathbb{N}$. Since $\Phi$ is  ultraweakly-wot continuous, then $\phi_{k}\in M_*$ is  normal. Besides
$$
\|\phi_{k}\| \leq \| \Phi \|\|x_{k}\|\|x_{k}^{*}\|,
\quad \sum_{k=1}^{\infty}\|\phi_{k}\| \leq \sum_{k=1}^{\infty} \|\Phi\|\|x_{k}\|\|x_{k}^{*}\|<+\infty.
$$
then we have $\sum_{k=1}^{\infty}\phi_{k}\in M_{*}$ is normal. Meanwhile
$$
\sum_{k=1}^{\infty}\phi_{k}=\sum_{k=1}^{\infty}x_{k}^{*}(\Phi(\cdot)x_{k})=\sum_{k=1}^{\infty} (J\Phi(\cdot))(x_k\otimes x_k^*)
=(J\Phi(\cdot))\Big(\sum_{k=1}^{\infty}x_{k}\otimes x_{k}^{*}\Big)=(J\Phi(\cdot))(u).
$$
Thus, $(J\Phi(\cdot))(u)$ is a normal functional on $M$.

$(3)\Rightarrow (2)$ is obvious,
since $x^*(\Phi(\cdot))x=(J\Phi(\cdot))(x\otimes x^*)$ is normal for any $x\in X$ and $x^*\in X^*$.
Thus, the proof is completed.
\end{proof}

\section{Elementary dilation space and dilation of quantum measures}


\begin{definition}\label{dilation1}
Let $U$ be a $B(X)$-valued quantum measure on a VN algebra $M$. If there exist a Banach space $Y$, bounded linear maps $S, T$ and a quantum measure $V: \mathcal{P}(M) \rightarrow B(Y)$ such that
$$
U(P)=S V(P) T
$$
then we say that $V$ is a dilation of $U$ and $Y$ is a dilation space.  We call $V$ a projection-valued dilation if $V(P)$ is an  idempotent on $Y$ for every $P \in \mathcal{P}(M)$.
\end{definition}


We remark that the dilation is not unique. In this paper we will be focused on the ``elementary dilation space" which was introduced for the commutative domain case in \cite{HLLL1}.
In order to construct such a dilation space, we will need the following lemma.

\begin{lemma}\cite{BW}\label{Gleasonthm}
Let $M$ be a VN algebra with no direct summand of type $I_{2}$  and let $X$ be a Banach space. Then each bounded, finitely additive $X$-valued measure $\mu$ on $\mathcal{P}(M)$ extends uniquely to a bounded linear operator $T$ from $M$ to $X$ with $\|\mu\|\le\|T\|\le 4\|\mu\|$.
\end{lemma}

\begin{remark}
In fact, the above Lemma \ref{Gleasonthm} is Bunce and Wright's  generalization of  original Gleason Theorem \cite{Gl}  for vector-valued measure version, called  the \textit{Generalized Gleason Theorem}. The proof strategy is to consider extending $\mu$ to a bounded map on the linear span $V(M)$ of $\mathcal{P}(M)$. Suppose that $ x=\sum_{i=1}^{n}\lambda_{i}P_{i} $ is a finite linear combination of projections $ \{P_{i}\}_{i=1}^{n}$  in $\mathcal{P}(M)$. Define a map $T: V(M) \rightarrow X$ by setting
$$
T\Big(\sum_{i=1}^{n} \lambda_{i} p_{i}\Big)=\sum_{i=1}^{n} \lambda_{i} \mu(p_{i}).
$$
It can be verified that  $T$ is well-defined  and the value depends only on $x$, and $\|T\| \leq 4\|\mu\|$. 
This allows us to extend $T$ uniquely to a bounded operator from $M$ into $X$.  Furthermore,  if complete additivity of $\mu$ is additionally assumed, then we have that  the extended mapping  $T$ is normal.
\end{remark}

With Lemma \ref{Gleasonthm} in hand,  now we  outline the algebraic structure of our dilation space.
Let $X$ be a Banach space and $M$ be a VN algebra without type $I_{2}$ direct summand, $\mathcal{P}(M)$ is  the projection lattice of $M$.  We use the symbol $\mathcal{L}(M, X)$ to denote the linear space of all vector-valued linear maps from $M$ to $X.$ Let $U$ be a $B(X)$-valued quantum measure on $\mathcal{P}(M).$ By the generalized Gleason Theorem, $U$ extends uniquely to a bounded linear operator $\overline{U}$ from $M$ to $B(X)$. For any $Q \in B_{M} $ and $x \in X,$ define
$$
\overline{U}_{Q, x}: M \rightarrow X, \quad \overline{U}_{Q, x}(R)=\overline{U}(RQ)x, ~~ \forall ~ R \in M.
$$
Then it is easy to see that $\overline{U}_{Q, x}$ is an $X$-valued linear map on $ M,$ that is, $\overline{U}_{Q, x} \in \mathcal{L}(M, X).$

Let $\mathcal{M}_{U}=\operatorname{span}\left\{\overline{U}_{Q, x}: Q \in B_{M}, x \in X \right\} \subset \mathcal{L}(M, X)$.
We will refer $\mathcal{M}_{U}$ as \textit{the elementary dilation space} induced by $U$.  Now, we have the following fundamental linear mappings from the algebraic structure. Define
\begin{equation}\label{equS}
S: \mathcal{M}_{U} \rightarrow X, \quad S(\Phi)=\Phi(I),  \quad  \forall ~  \Phi \in \mathcal{M}_{U}. \\
\end{equation}
\begin{equation}\label{equT}
T: X \rightarrow \mathcal{M}_{U}, \quad T(x)=\overline{U}_{I, x}, \quad  \forall ~  x\in X.
\end{equation}
and let  $\{C_{i}\}_{i=1}^{N}\subset \mathbb{C}, \{Q_{i}\}_{i=1}^{N}\subset B_{M}, \{x_{i}\}_{i=1}^{N}\subset X$  and $\sum_{i=1}^{N}C_{i}\overline{U}_{Q_{i},x_{i}} \in \mathcal{M}_{U}$, define
\begin{equation}\label{equV}
 V: \mathcal{P}(M) \to \mathcal{L}(\mathcal{M}_{U}), \quad
V(P)\Big(\sum_{i=1}^{N}C_{i}\overline{U}_{Q_{i},x_{i}}\Big)
=\sum_{i=1}^{N}C_{i}\overline{U}_{PQ_{i},x_{i}} ~\forall P \in  \mathcal{P}(M).
\end{equation}
where $\mathcal{L}(\mathcal{M}_{U})$ denotes the space of all  linear maps on $\mathcal{M}_{U}$. Actually for any $\Phi \in \mathcal{M}_{U}$,
$$
\Big( V(P)(\Phi)\Big)(R)=\Phi(RP), \forall ~ R \in M
$$
which is independent of representations of $\Phi$, thus $V(P)$ is well-defined in $\mathcal{L}(\mathcal{M}_{U}).$  Moreover
$$
V(P)V(P)\Big(\sum_{i=1}^{N}C_{i}\overline{U}_{Q_{i},x_{i}}\Big)
=\sum_{i=1}^{N}C_{i}\overline{U}_{P\cdot PQ_{i},x_{i}}
=\sum_{i=1}^{N}C_{i}\overline{U}_{PQ_{i},x_{i}}
=V(P)\Big(\sum_{i=1}^{N}C_{i}\overline{U}_{Q_{i},x_{i}}\Big).
$$

Let  $P_{1}, P_{2} \in \mathcal{P}(M)$ such that $P_{1}\perp P_{2}$,  then $P_{1}+P_{2}=P_{1}\vee P_{2}\in \mathcal{P}(M)$ and
$$
\begin{aligned}
V(P_{1}+P_{2})\Big(\sum_{i=1}^{N}C_{i}\overline{U}_{Q_{i},x_{i}}\Big)
&=\sum_{i=1}^{N}C_{i}\overline{U}_{(P_{1}+P_{2})Q_{i},x_{i}}\\
&=\sum_{i=1}^{N}C_{i}\overline{U}_{P_{1}Q_{i},x_{i}}
+\sum_{i=1}^{N}C_{i}\overline{U}_{P_{2}Q_{i},x_{i}}\\
&=\left( V(P_{1})+V(P_{2})\right)\Big(\sum_{i=1}^{N}C_{i}\overline{U}_{Q_{i},x_{i}}\Big).
\end{aligned}
$$
Thus $V: \mathcal{P}(M)\to  \mathcal{L}(\mathcal{M}_{U}) $ is finitely additive and  idempotent valued.

Obviously we have
$$
U(P)=S V(P) T, ~\forall P\in \mathcal{P}(M)
$$

Observe that $V$ is only an algebraic dilation of $U$.  So we  need to impose a Banach space structure into consideration.  

\begin{definition}
Let $\mathcal{M}_{U}$ be the elementary dilation space induced by $U$. Assume that $\|\cdot\|_{V}$ is a norm on $\mathcal{M}_{U},$ and denote its completion by $\widetilde{\mathcal{M}}_{U}.$  Then the norm $\|\cdot\|_{V}$ on $\mathcal{M}_{U}$ is called a \textit{dilation norm} of $U$ if the above linear mappings $S, T$ and $V(P)$ for any  $P \in \mathcal{P}(M)$ all are bounded.
\end{definition}

While the dilation space is not unique, the following result establishes the connections between any  projection-valued quantum measure dilation space  and  the elementary dilation space: any given  projection-valued dilation space $Y$ can induce a dilation norm  on $\mathcal{M}_{U}$ so that  $U$ can be dilated to a projection-valued quantum measure from $\mathcal{P}(M)$ to $B(\widetilde{\mathcal{M}}_{U})$.

\begin{theorem} \label{thm3.1}
Let $X$ be a Banach space and $M$ be a VN algebra without type $I_{2}$ direct summand. Let $U$ be  a $B(X)$-valued quantum measure on $\mathcal{P}(M)$ and $V$ be a corresponding projection-valued quantum measure from $\mathcal{P}(M)$ to $B(Y)$, and $S: Y\to X$, $ T: X \to Y $ are bounded linear operators. Then there exist an induced dilation norm $\|\cdot\|_{\mathcal{D}}$ on $\mathcal{M}_{U}$ (denote its completion by $\widetilde{\mathcal{M}}_{U,\mathcal{D}}$),  bounded linear maps $S_{\mathcal{D}}, T_{\mathcal{D}}$  and a projection-valued quantum measure $V_{\mathcal{D}}: \mathcal{P}(M)\to B(\widetilde{\mathcal{M}}_{U,\mathcal{D}})$ such that
$$
U(P)=S_{\mathcal{D}}V_{\mathcal{D}}(P) T_{\mathcal{D}}
$$
Furthermore,  there exists a linear  contraction
$$
W: \widetilde{\mathcal{M}}_{U,\mathcal{D}} \to Y/\ker S
$$
such that $S_{\mathcal{D}}=SW,~~ WT_{D}=qV(1)T$,  where $q$ is the quotient map form $Y$ to $Y/\ker S$ by $q(x)=[x]$.
\end{theorem}

\begin{proof}
By Lemma \ref{Gleasonthm},  the quantum measure  $V$ extends to a bounded linear operator $\overline{V}$ from $M$ to $B(Y)$. Define $\|\cdot\|_{\mathcal{D}}: \mathcal{M}_{U}\to \mathbb{R}^{+}\cup \{0\}$ by
$$
\Big\|\sum_{i=1}^{N}C_{i}\overline{U}_{Q_{i},x_{i}}\Big\|_{\mathcal{D}}=\sup_{R\in B_{M}}\Big\|\Big[\sum_{i=1}^{N}C_{i}
\overline{V}(RQ_{i})Tx_{i}\Big]\Big\|_{Y/\ker S}
$$
where $N>0$,  $\{C_{i}\}_{i=1}^{N}\subset \mathbb{C}, \{Q_{i}\}_{i=1}^{N}\subset B_{M}, \{x_{i}\}_{i=1}^{N}\subset X$,  $\ker S$ is the kernel of operator $S$ and $\|\cdot\|_{Y/\ker S}$ denotes quotient norm on $Y/\ker S$.

First, we show  that $\|\cdot\|_{\mathcal{D}}$ is well-defined. If we have two representations of $\Phi \in \mathcal{M}_{U}$,
$$
\Phi=\sum_{i=1}^{N}C_{i}\overline{U}_{Q_{i},x_{i}}=\sum_{j=1}^{M}D_{j}\overline{U}_{Q_{j}^{\prime},y_{j}},
$$
where $M>0$,  $\{D_{j}\}_{j=1}^{M}\subset \mathbb{C}, \{Q_{j}^{\prime}\}_{j=1}^{M}\subset B_{M}, \{y_{j}\}_{j=1}^{M}\subset X$. Applying the generalized Gleason Theorem (see Lemma \ref{Gleasonthm}), we obtain that $ \overline{U}(\cdot)=S\overline{V}(\cdot)T. $
Thus  for $R \in M$
$$
\Phi(R)=S\Big(\sum_{i=1}^{N}C_{i}\overline{V}(RQ_{i})Tx_{i}\Big)
=S\Big(\sum_{j=1}^{M}D_{j}\overline{V}(RQ_{j}^{\prime})Ty_{j}\Big).
$$
Thus $\|\Phi\|_{\mathcal{D}}$ is  independent of the representation of $\Phi$. And
$$
\begin{aligned}
\Big\|\sum_{i=1}^{N}C_{i}\overline{U}_{Q_{i},x_{i}}\Big\|_{\mathcal{D}}
&=\sup_{R\in B_{M}}\Big\|\Big[\sum_{i=1}^{N}C_{i}
\overline{V}(RQ_{i})Tx_{i}\Big]\Big\|_{Y/\ker S} \\
&\leq \sup_{R\in B_{M}}\sum_{i=1}^{N}|C_{i}|
\big\|\big[\overline{V}(RQ_{i})Tx_{i}\big]\big\|_{Y/\ker S}\\
&\leq \sup_{R\in B_{M}}\sum_{i=1}^{N}|C_{i}|
\big\|\overline{V}(RQ_{i})Tx_{i}\big\|_{Y}\\
&\leq 4\|V\|\|T\|\Big(\sum_{i=1}^{N}|C_{i}|\|x_{i}\|\Big).
\end{aligned}
$$

Next we divide the rest of the proof into four steps:

(i)  $\|\cdot\|_{\mathcal{D}}$ is a norm. Obviously  for any $\Phi\in \mathcal{M}_{U}$, $\|\Phi\|_{\mathcal{D}}\geq 0$. If $\big\| \Phi\big\|_{\mathcal{D}}=0$,  suppose $\Phi=\sum_{i=1}^{N}C_{i}\overline{U}_{Q_{i},x_{i}}$, then,  $\sup_{R\in B_{M}}\big\|\big[\sum_{i=1}^{N}C_{i}
\overline{V}(RQ_{i})Tx_{i}\big]\big\|_{Y/\ker S}=0$, then for any $R\in B_{M}$, $\big[\sum_{i=1}^{N}C_{i}\overline{V}(RQ_{i})Tx_{i}\big]=0$ in  $Y/\ker S$, so $S\big(\sum_{i=1}^{N}C_{i}\overline{V}(RQ_{i})Tx_{i}\big)=0$, namely, $\Phi(R)=0 $ for any  $R$ in $ B_{M}$. By linearity, it follows that $\Phi=0.$  And it is routine to verify that  homogeneity and triangle inequality.  Thus  we have proved that $\|\cdot\|_{\mathcal{D}}$ is a norm on $\mathcal{M}_{U}$.

%

(ii)  Let $\widetilde{\mathcal{M}}_{U,\mathcal{D}}$ be the  completion of the normed space $\mathcal{M}_{U}$. We verify next the boundedness of $S_{\mathcal{D}}, T_{\mathcal{D}} $.  Since the linear mapping $S_{\mathcal{D}}: \widetilde{\mathcal{M}}_{U,\mathcal{D}} \to X$ is well-defined by Equation (\ref{equS}), then we have
$$
S_{\mathcal{D}}\Big(\sum_{i=1}^{N}C_{i}\overline{U}_{Q_{i},x_{i}}\Big)=\sum_{i=1}^{N}C_{i}\overline{U}(Q_{i})x_{i},
$$
and
$$
\begin{aligned}
\Big\|S_{\mathcal{D}}(\sum_{i=1}^{N}C_{i}\overline{U}_{Q_{i},x_{i}})\Big\|_{X}
&=\Big\|\sum_{i=1}^{N}C_{i}S\overline{V}(Q_{i})Tx_{i}\Big\|_{X}\\
&\leq \|S\|\Big\|\Big[\sum_{i=1}^{N}C_{i}V(P_{i})Tx_{i}\Big] \Big\|_{Y/\ker S}\\
&\leq \|S\|\Big\|\sum_{i=1}^{N}C_{i}\overline{U}_{Q_{i},x_{i}}\Big\|_{\mathcal{D}}.
\end{aligned}
$$

And  the linear mapping $T_{\mathcal{D}}:  X\to \widetilde{\mathcal{M}}_{U, \mathcal{D}}; ~  T_{\mathcal{D}}(x)=\overline{U}_{I,x}$ is well-defined by Equation (\ref{equT}), then we have
$$
\|T_{\mathcal{D}}(x)\|=\sup_{R\in B_{M}}\left\| [ \overline{V}(R)Tx ]\right\|_{Y/\ker S}
\leq \sup_{R\in B_{M}}\|\overline{V}(R)Tx\|_{Y}
\leq 4\|V\|\|T\|\|x\|.
$$
Thus $ S_{\mathcal{D}}, T_{\mathcal{D}}$ are both  bounded linear operators.

(iii) By Equation (\ref{equV}),  the mapping $V_{\mathcal{D}}: \mathcal{P}(M)\to B(\widetilde{\mathcal{M}}_{U,\mathcal{D}})$  is  finitely additive and  idempotent valued.



%


Moreover  for every $P\in \mathcal{P}(M)$,
$$
\begin{aligned}
\Big\|V_{\mathcal{D}}(P)\Big(\sum_{i=1}^{N} C_{i}\overline{U}_{Q_{i}, x_{i}}\Big)\Big\|_{\mathcal{D}}
&=\Big\|\sum_{i=1}^{N}C_{i}\overline{U}_{PQ_{i},x_{i}}\Big\|_{\mathcal{D}}\\
&=\sup_{R\in B_{M}}\Big\|\Big[\sum_{i=1}^{N}C_{i}\overline{V}(RPQ_{i})Tx_{i}\Big]\Big\|_{Y/\ker S}\\
&\leq \Big\|\sum_{i=1}^{N} C_{i}\overline{U}_{Q_{i}, x_{i}}\Big\|_{\mathcal{D}}\\
\end{aligned}
$$
It follows that for any $P\in \mathcal{P}(M)$,  $\|V_{\mathcal{D}}(P)\|\leq 1$. Thus  the map $V_{\mathcal{D}}$ is well-defined and is a  projection-valued  quantum measure.

(iv) Finally  we show that there exists a linear  contraction from $\widetilde{\mathcal{M}}_{U,\mathcal{D}}$ into $Y$. Define $W:\widetilde{\mathcal{M}}_{U,\mathcal{D}} \to Y/\ker S $ by
$$
W:\sum_{i=1}^{N} C_{i}\overline{U}_{Q_{i},x_{i}} \to \Big[\sum_{i=1}^{N} C_{i}\overline{V}(Q_{i})Tx_{i}\Big].
$$
It is easy to see that $W$ is a well-defined linear contraction. Moreover
$$
\begin{aligned}
S_{\mathcal{D}}\Big(\sum_{i=1}^{N} C_{i}\overline{U}_{Q_{i},x_{i}}\Big)
&=\sum_{i=1}^{N} C_{i}\overline{U}(Q_{i})x_{i}=\sum_{i=1}^{N} C_{i}S\overline{V}(Q_{i})Tx_{i}\\
&=S\Big(\sum_{i=1}^{N} C_{i}\overline{V}(Q_{i})Tx_{i}\Big)=SW\Big(\sum_{i=1}^{N} C_{i}\overline{U}_{Q_{i},x_{i}}\Big)
\end{aligned}
$$
and for any $x\in X$,
$$
WT_{\mathcal{D}} x=W(\overline{U}_{I,x})=[V(I)T x]=q V(I)T x
$$
Thus we have $S_{\mathcal{D}}=SW,~~ WT_{D}=qV(I)T.$
\end{proof}


In what follows we will focus on introducing some special norms (such a norm will be called a {\it dilation norm})  on the (algebraic) elementary dilation space $\mathcal{M}_{U}$ so that its completion provides us the needed projection-valued dilations. Define $\|\cdot\|_{\mathcal{E}}:\mathcal{M}_{U}\to \mathbb{R}^{+}\cup 0 $ by
$$
\Big\|\sum_{i=1}^{N}C_{i}\overline{U}_{Q_{i},x_{i}}\Big\|_{\mathcal{E}}=\sup_{R\in B_{M}}\Big\|\sum_{i=1}^{N}C_{i}\overline{U}(RQ_{i})x_{i}\Big\|_{X}
$$
where integer $N>0$, $\{C_{i}\}_{i=1}^{N}\subset \mathbb{C}, \{Q_{i}\}_{i=1}^{N}\subset B_{M}, \{x_{i}\}_{i=1}^{N}\subset X$.

Assume that $\Psi \in \mathcal{M}_{U}$ has two different representations, set
$$
\Psi=\sum_{i=1}^{N}C_{i}\overline{U}_{Q_{i},x_{i}}=\sum_{j=1}^{M}D_{j}\overline{U}_{Q_{j}^{\prime},y_{j}},
$$
where  $ \{C_{i}\}_{i=1}^{N},\{D_{j}\}_{j=1}^{M} \subset \mathbb{C}, \{Q_{i}\}_{i=1}^{N},\{Q_{j}^{\prime}\}_{j=1}^{M} \subset B_{M}, \{x_{i}\}_{i=1}^{N},\{y_{j}\}_{j=1}^{M} \subset X$. Then for any $R\in M$, we have
$$
\Psi(R)=\sum_{i=1}^{N}C_{i}\overline{U}(RQ_{i})x_{i}=\sum_{j=1}^{M}D_{j}\overline{U}(RQ_{j}^{\prime})y_{j}.
$$
Then  $\|\Psi\|_{\mathcal{E}}$ is independent of the  representation of $\Psi$. And
$$
\|\Psi\|_{\mathcal{E}}=\Big\|\sum_{i=1}^{N}C_{i}\overline{U}_{Q_{i},x_{i}}\Big\|_{\mathcal{E}}=\sup_{R\in B_{M}}\Big\|\sum_{i=1}^{N}C_{i}\overline{U}(RQ_{i})x_{i}\Big\|_{X}\leq 4\|U\|\Big(\sum_{i=1}^{N}|C_{i}|\|x_{i}\|\Big) .
$$
Thus $\|\cdot\|_{\mathcal{E}}$ is well-defined.

\begin{proposition}\label{dilationnorm}
$\|\cdot\|_{\mathcal{E}}$ is a dilation norm on $\mathcal{M}_{U}$.
\end{proposition}

\begin{proof}
We divide our proof in two steps.

(i) First, we show $\|\cdot\|_{\mathcal{E}}$ is a norm. Obviously $\|\Psi\|_{\mathcal{E}}\geq 0, ~ \forall ~\Psi\in \mathcal{M}_{U}$.
If $\|\sum_{i=1}^{N}C_{i}\overline{U}_{Q_{i},x_{i}}\|_{\mathcal{E}}=0$, then for any $R\in B_{M}$, $\sum_{i=1}^{N}C_{i}\overline{U}(RQ_{i})x_{i}=0$, i.e., $ \big(\sum_{i=1}^{N}C_{i}\overline{U}_{Q_{i},x_{i}}\big)(R)=0$.  By linearity, then we have
$
\sum_{i=1}^{N}C_{i}\overline{U}_{Q_{i},x_{i}}=0.
$

By a routine calculation, we obtain that for any $\Phi, \Psi  \in \mathcal{M}_{U}$,
$$
\|\lambda \Psi\|_{\mathcal{E}}=|\lambda|\|\Psi\|_{\mathcal{E}} ~\forall ~ \lambda \in \mathbb{C} \quad  \text{ and }  \quad \|\Phi+ \Psi\|_{\mathcal{E}}\leq \|\Phi \|_{\mathcal{E}}+ \|\Psi\|_{\mathcal{E}}
$$
Hence $\|\cdot\|_{\mathcal{E}}$  is indeed a norm on $\mathcal{M}_{U}$.

(ii) Now we show that $\|\cdot\|_{\mathcal{E}}$ is a dilation  norm of $U$.

Note that the linear mapping $S_{\mathcal{E}}: \widetilde{\mathcal{M}}_{U, \mathcal{E}} \to X, S_{\mathcal{E}}(\Phi)=\Phi(I)$ is well-defined by Equation (\ref{equS}), and
$$
\Big\|S_{\mathcal{E}}\Big(\sum_{i=1}^{N}C_{i}\overline{U}_{Q_{i},x_{i}}\Big)\Big\|_{X}
=\Big\|\sum_{i=1}^{N}C_{i}\overline{U}(Q_{i})x_{i} \Big\|_{X}
\leq \sup_{R\in B_{M}}\Big\|\sum_{i=1}^{N}C_{i}\overline{U}(RQ_{i})x_{i} \Big\|_{X}
=\Big\| \sum_{i=1}^{N}C_{i}\overline{U}_{Q_{i},x_{i}}\Big\|_{\mathcal{E}}.
$$
Observed that we have verified  that $S_{\mathcal{E}}$ is bounded on the dense subspace of $\widetilde{\mathcal{M}}_{U, \mathcal{E}}$.  Thus $S_{\mathcal{E}}$ is bounded.

Meanwhile the linear mapping  $T_{\mathcal{E}}: X \to \widetilde{\mathcal{M}}_{U, \mathcal{E}}; ~ T_{\mathcal{E}}x=\overline{U}_{I,x}, ~\forall  x\in X$  is well-defined by Equation  (\ref{equT}), then we have
$$
\|Tx\|_{\mathcal{E}}=\|\overline{U}_{I,x}\|_{\mathcal{E}}=\sup_{R\in B_{M}}\|\overline{U}(R\cdot I)x\|\leq \sup_{R\in B_{M}}\|\overline{U}(R)\|\|x\|\leq 4\|U\|\|x\|.
$$
Hence $S_{\mathcal{E}},T_{\mathcal{E}} $ are both  bounded such that  $\|S_{\mathcal{E}}\|\leq 1$ and $\|T_{\mathcal{E}}\|\leq 4\|U\|. $

Besides,  the mapping $V_{\mathcal{E}}: \mathcal{P}(M)\to B(\widetilde{\mathcal{M}}_{U, \mathcal{E}})$ defined by Equation (\ref{equV}) is finitely additive and idempotent valued. Moreover, for any $P\in \mathcal{P}(M)$,
$$
\begin{aligned}
\Big\|V_{\mathcal{E}}(P)\Big(\sum_{i=1}^{N} C_{i}\overline{U}_{Q_{i}, x_{i}} \Big) \Big\|_{\mathcal{E}}
&=\Big\|\sum_{i=1}^{N}C_{i}\overline{U}_{PQ_{i},x_{i}} \Big\|_{X}\\
&=\sup_{R\in B_{M}}\Big\|\sum_{i=1}^{N}C_{i}\overline{U}(RPQ_{i})Tx_{i}\Big\|_{X}\\
&\leq \Big\|\sum_{i=1}^{N} C_{i}\overline{U}_{Q_{i}, x_{i}}\Big\|_{\mathcal{E}}.
\end{aligned}
$$
which implies that $V_{\mathcal{E}}(P) $ is bounded on the normed space $\mathcal{M}_{U}$, and  hence bounded on its completion $\widetilde{\mathcal{M}}_{U, \mathcal{E}}$

Therefore,  we have proved that $\widetilde{\mathcal{M}}_{U, \mathcal{E}}$ is a dilation space and $V_{\mathcal{E}}$ is a projection-valued dilation  of $U$.
\end{proof}



Now comes to our first main result on the dilation of countably additive quantum measures. Recall  that a Banach space is said to have the Schur property if every weakly convergent sequence converges in norm. (cf. \cite{AK})

\begin{theorem}\label{countablemeasuredilation}
Let $M$ be a $\sigma$-finite VN algebra without type $I_{2}$ direct summand,   and  assume that  a  Banach space X has either Schur property or it is the $\ell_{p}( 1\leq p <2 )$ space.  If
$$
U:\mathcal{P}(M)\to B(X),
$$
is  a countably additive  operator-valued quantum measure, then there exist a Banach space $Z$, bounded linear operators $S: Z\to X $ and $T: X \to Z $, and a countably additive  projection-valued  quantum measure $V: \mathcal{P}(M) \to B(Z) $ such that
$$
U(P)=SV(P)T.
$$
\end{theorem}

\begin{proof}

Applying similar  arguments as in the proof of  Proposition \ref{dilationnorm}, we define
$$
\mathcal{M}_{U, X}=\operatorname{span}\{ \overline{U}_{Q,x} \big| Q\in B_{M}, x\in X \}
$$
and let
$
\widetilde{\mathcal{M}}_{U, X}
$
be its completion under the norm
$$
\Big\|\sum_{i=1}^{N}C_{i}\overline{U}_{Q_{i},x_{i}}\Big\|=\sup_{R\in B_{M}}\Big\|\sum_{i=1}^{N}C_{i}\overline{U}(RQ_{i})x_{i}\Big\|_{X},
$$
where $\sum_{i=1}^{N}C_{i}\overline{U}_{Q_{i},x_{i}}\in \mathcal{M}_{U, X} $ and $\overline{U}$ is ultraweakly-wot continuous, bounded linear extension of $U$.
Then the following two linear operators
$$
S: \widetilde{\mathcal{M}}_{U, X} \to X, \quad  S(\Phi)=\Phi(I).
\quad \quad
T: X \to \widetilde{\mathcal{M}}_{U, X},\quad  T(x)=\overline{U}_{I,x}
$$
are all bounded. Furthermore
$$
V: \mathcal{P}(M) \to B(\widetilde{\mathcal{M}}_{U, X}), \quad  V(P)\Big(\sum_{i=1}^{N}C_{i}\overline{U}_{Q_{i},x_{i}}\Big)
=\sum_{i=1}^{N}C_{i}\overline{U}_{PQ_{i}, x_{i}}, ~\forall~ P\in \mathcal{P}(M).
$$
is a projection-valued dilation of $U$.

To prove the countable additivity of $V$, by Remark \ref{Remarkondef} on the  Definition  \ref{quantummeasure} about quantum measures, it is equivalent to verify that $V$ is strong countably additive on $\widetilde{\mathcal{M}}_{U, X}$. Since we have proved that  $V$ is uniform bounded ($\|V(P)\|\leq 1, \forall P\in  \mathcal{P}(M)$) on $\mathcal{M}_{U, X}$, it is enough to verify that $V$ is strong countably additive on $\mathcal{M}_{U, X}$. Indeed,  for  $x\in \widetilde{\mathcal{M}}_{U, X}$,  there exists a sequence $\{x_{i}\}_{i=1}^{\infty}\subset \mathcal{M}_{U, X}$ such that $ \lim_{i\to \infty }x_{i}=x $.   Then for any  countable collection of mutually orthogonal projections $\{P_{n}\}_{n=1}^{\infty}$ of $\mathcal{P}(M)$ with supremum $P$, we have that
\begin{equation}\label{strongadditive}
\begin{aligned}
&\Big \|V(P)x -\sum_{n=1}^{N}V(P_{n})x\Big\|\\
\leq & \| V(P) \| \|x-x_{M}\|  + \Big\| V(P)x_{M}- \sum_{n=1}^{N}V(P_{n})x_{M}\Big\| + \Big\| W\Big(\sum_{n=1}^{N}P_{n}\Big)\Big\|\left\|x-x_{M}\right\|\\
\leq & 2 \|x-x_{M}\|+ \Big\| V(P)x_{M}- \sum_{n=1}^{N}V(P_{n})x_{M}\Big\|,
\end{aligned}
\end{equation}
where $x_{M}\in \mathcal{M}_{U, X}$. If we have proved that $V$ is strong countably additive on $\mathcal{M}_{U, X}$,  then for  any $\varepsilon>0,$ we can find $M>0$ and $L>0$ such that $\|x- x_M\| < \varepsilon$ and
$ \big\| V(P)x_{M}- \sum_{n=1}^{N}V(P_{n})x_{M}\big\| \leq \varepsilon$ for every $N \geq L$, which imply that  $V$ is countably additive  on $\widetilde{\mathcal{M}}_{U, X}$.



Now we show $V$ is strong countably additive on $\mathcal{M}_{U, X}$. Let $\{P_{j}\}_{j=1}^{\infty}$ be a countable family of  mutually orthogonal projections in $\mathcal{P}(M)$ with supremum $P$.  Then for any  $\sum_{i=1}^{N}C_{i}\overline{U}_{Q_{i},x_{i}}\in \mathcal{M}_{U, X},$  we have

$$
\begin{aligned}
&\Big\|V(P)\Big(\sum_{i=1}^{N} C_{i} \overline{U}_{Q_{i}, x_{i}}\Big)-\sum_{j=1}^{M} V(P_{j})\Big(\sum_{i=1}^{N} C_{i} \overline{U}_{Q_{i}, x_{i}}\Big)\Big\|\\
=&\Big\|\sum_{i=1}^{N}C_{i} \Big( \overline{U}_{PQ_{i},x_{i}}
-\sum_{j=1}^{M}\overline{U}_{P_{j}P_{i},x_{i}}\Big)\Big\|\\
=& \sup_{R\in B_{M}} \Big\| \sum_{i=1}^{N}C_{i} \Big(\overline{U}(RPQ_{i})x_{i} -\sum_{j=1}^{M}\overline{U}(RP_{j}Q_{i})x_{i}\Big)\Big\|_{X}\\
=&\sup_{R\in B_{M}} \Big\|\sum_{i=1}^{N}C_{i} \Big(\overline{U}(R(\sum_{j=M+1}^{\infty}P_{j})Q_{i})\Big)x_{i} \Big\|_{X}\\
\leq&\sum_{i=1}^{N}\Big|C_{i}\Big| \sup _{R^{\prime}\in B_{M}}
\Big\|\overline{U}\Big(\sum_{j=M+1}^{\infty}(R^{\prime}P_{j}Q_{i})\Big) x_{i}\Big\|_{X}\\
\end{aligned}
$$

\vspace{3mm}
\noindent{\bf Case I. } $X$ has Schur property.  Note that $\sum_{j=1}^{\infty}P_{j}\stackrel{sot}{\longrightarrow} P$ implies  $\sum_{j=M+1}^{\infty}P_{j}\stackrel{sot}{\longrightarrow} 0$ as $M\to \infty $. Then for any $R\in B_{M}$ we have $\sum_{j=M+1}^{\infty}RP_{j}Q_{i}\stackrel{wot}{\longrightarrow} 0 $. Since the ultraweak topology and the weak operator topology coincide in a bounded set, we have $\sum_{j=M+1}^{\infty}RP_{j}Q_{i}\stackrel{w*}{\longrightarrow} 0 $, which implies by  the ultraweak-wot continuity of $\overline{U}$ that
$$
\overline{U}\Big(\sum_{j=M+1}^{\infty}RP_{j}Q_{i}\Big)x_{i} \stackrel{w}{\longrightarrow} 0, \quad M \to \infty.
$$

Since $X$ has Schur property, we obtain that
$$
\Big\|\overline{U}\Big(\sum_{j=M+1}^{\infty}(RP_{j}Q_{i})\Big) x_{i}\Big\|_{X}\longrightarrow  0, ~ \forall ~ R\in B_{M},
$$
which implies the  strong countable additivity on $\mathcal{M}_{U, X}$.

\vspace{3mm}
\noindent{\bf Case II. $X = \ell_{p}$ ($1\leq p <2$).} Assume to the contrary that  $\sup _{R^{\prime} \in B_{M}}\big\|\overline{U}\big(\sum_{j=M+1}^{\infty}R^{\prime}P_{j}Q_{i}\big) x_{i}\big\|_{l_{p}}$ does not go to $0$ as $M\to \infty $.  Then we can find a $\delta > 0$, a sequence of $n_{1}\leq m_{1}< n_{2}\leq m_{2} < n_{3}\leq m_{3}< \dots $ and $\{   R_{l}^{\prime} \}_{l=1}^{\infty} \subset B_{M}$  such that
\begin{equation}\label{contradiction1}
\Big\|\sum_{j=n_{l}}^{m_{l}}\overline{U}(R_{l}^{\prime}P_{j}^{\prime}Q_{i})x_{i}\Big\|\geq \delta,  \forall ~ l\in \mathbb{N}
\end{equation}

Set $P_{l}^{\prime}=\sum_{j=n_{l}}^{m_{l}}P_{j}$, and $y_{l}=\overline{U}(R_{l}^{\prime}P_{l}^{\prime}Q_{i})x_{i}$. Then we have two claims about the sequence $\{ y_{l}\}_{l=1}^{\infty}$ which will lead to a contradiction.
\vspace{3mm}

\noindent \textbf{Claim 1}: $ \forall ~ (\alpha_{l})\in \ell_{2}, ~~ \sum_{l=1}^{\infty}\alpha_{l}y_{l}$  converges unconditionally in norm

To verify claim 1,  we first show that  for any $(\alpha_{l})\in \ell_{2}$,
$
\sum_{l=1}^{\infty} \alpha_{l} R_{l}^{\prime}P_{l}^{\prime}Q_{i}
$
converges  unconditionally in norm,  where $ R_{l}^{\prime}\in B_{M} $  and $\{P_{l}^{\prime}\}$ is also a family of mutually orthogonal projections.  By the characterization of unconditional convergence in \cite{He}, this is equivalent to show that $\sum_{l=1}^{\infty}\varepsilon_{l}\alpha_{l} R_{l}^{\prime}P_{l}^{\prime}Q_{i}$  converges for every choice of signs $ \varepsilon_{l}=\pm 1.$

%

As any $\sigma$-finite VN algebra can be realized for some Hilbert space $H$ as a VN subalgebra  in $ B(H)$ \cite{Pi}, Note that for every $ N>M $, we have
$$
\begin{aligned}
\Big\|\sum_{l=M+1}^{N}\varepsilon_{l}\alpha_{l} R_{l}^{\prime}P_{l}^{\prime}Q_{i}\Big\|&=\sup_{x\in H,\|x\|\leq 1} \Big|\Big\langle \sum_{l=M+1}^{N}\varepsilon_{l}\alpha_{l} R_{l}^{\prime}P_{l}^{\prime}Q_{i}x, x \Big\rangle\Big|\\
&\leq \sup_{x\in H,\|x\|\leq 1}\Big(\sum_{l=M+1}^{N}|\alpha_{l}|
\|P_{l}^{\prime}Q_{i}x\|\|{R_{l}^{\prime}}^{*}x\| \Big)\\
&\leq \sup_{x\in H,\|x\|\leq 1}\Big(\sum_{l=M+1}^{N}|\alpha_{l}|
\|P_{l}^{\prime}Q_{i}x\|\Big)\\
&\leq \sup_{x\in H,\|x\|\leq 1} \Big(\sum_{l=M+1}^{N}|\alpha_{l}|^{2}\Big)^{1/2}
\Big(\sum_{l=M+1}^{N}\|P_{l}^{\prime}Q_{i}x\|^{2} \Big)^{1/2} \\
&\leq \Big( \sum_{l=M+1}^{N}|\alpha_{l}|^{2}\Big)^{1/2} \Big(\sup_{x\in H,\|x\|\leq 1} \|Q_{i}x\|\Big)\\
&\leq \Big( \sum_{l=M+1}^{N}|\alpha_{l}|^{2}\Big)^{1/2}.
\end{aligned}
$$


Since $( \alpha_{l} )\in \ell_{2} $, the above inequality implies that $\sum_{l=1}^{\infty} \alpha_{l} R_{l}^{\prime}P_{l}^{\prime}Q_{i} $  is unconditionally convergent in norm and  hence every subseries converges in the sense that $\sum_{k=1}^{\infty}  \alpha_{l_{k}} R_{l_{k}}^{\prime}P_{l_{k}}^{\prime}Q_{i} $ converges for every increasing sequence $\{\ell_{k}\}$.  Since  $\overline{U}$ is  ultraweakly-wot continuous,  we get that  every subseries of $\sum_{l=1}^{\infty} \alpha_{l}\overline{U}(R_{l}^{\prime}P_{l}^{\prime}Q_{i})x_{i}$ is weakly convergent. Thus, by the  Orlicz-Pettis Theorem, we get that $\sum_{l=1}^{\infty} \alpha_{l}\overline{U}( R_{l}^{\prime}P_{l}^{\prime}Q_{i})x_{i}$ converges unconditionally. Therefore we have proved the claim that $ \sum_{k=1}^{\infty}\alpha_{l}y_{l}$ converges unconditionally  in norm for any $(\alpha_{l})\in \ell_{2}$.


The second claim is about the subsequence of $\{ y_{l}\}_{l=1}^{\infty}$. We  first introduce some necessary terminology in Banach space theory. We say that  sequence $\{ x_{n}\}_{n=1}^{\infty}$ is a \textit{basic sequence} if $\{ x_{n}\}_{n=1}^{\infty}$ is a (Schauder) basis for its closed linear span. Two basic sequences $\{ x_{n}\}_{n=1}^{\infty}$ and $\{ y_{n}\}_{n=1}^{\infty}$ are said to be \textit{equivalent} if for every sequence of scalars $\{ a_{n}\}_{n=1}^{\infty}$, $\sum_{n=1}^{\infty}a_{n}x_{n}$ converges if and only if $\sum_{n=1}^{\infty}a_{n}y_{n}$ converges. Given a basic sequence $(x_{n})_{n=1}^{\infty}$ and an increasing sequence of positive integer $p_{1} < q_{1} < p_{2} < q_{2} \ldots.$  If  $z_{k}= \sum_{i=p_{k}}^{q_{k}}b_{i}x_{i}$ is a nonzero vector for each $k$, then we say $(z_{k})_{k=1}^{\infty}$ is a \textit{block basic sequence} with respect to $\{ x_{n}\}_{n=1}^{\infty}$.

\vspace{3mm}
\noindent \textbf{Claim 2}: There exists a subsequence $\{ y_{l_{k}}\}_{k=1}^{\infty}$ of $ \{ y_{l} \}_{l=1}^{\infty}$  that is equivalent to the canonical unit vector basis $\{ e_{k} \}_{k=1}^{\infty}$ of $\ell_{p} (1\leq p < 2).$

Since $ \{  P_{l}^{\prime} \}_{l=1}^{\infty} $ is a family of mutually orthogonal projections such that $\sum_{l=1}^{\infty}P_{l}^{\prime}\leq P$,   we get that $P_{l}^{\prime}\stackrel{sot}{\longrightarrow} 0$ as $l\to \infty $. Then for any operator sequence $\{R_{l}^{\prime}\}\subset B_{M}$, we have
$$
R_{l}^{\prime}P_{l}^{\prime}Q_{i} \stackrel{w*}{\longrightarrow} 0, \quad l\to \infty,
$$
and hence
$$
y_{l}=\overline{U}(R_{l}^{\prime}P_{l}^{\prime}Q_{i})x_{i}\stackrel{w}{\longrightarrow} 0, \quad l \to \infty.
$$

%
%
%
%
%


Since by (\ref{contradiction1}), there exists a $\delta>0$ such that for all $l$, $\|y_{l}\|\geq \delta$, let $0<\beta< \frac{1}{7}$ and  $S_{m}$ denote the $m$th-partial sum operator with respect  to the canonical basis $\{e_{n}\}_{n=1}^{\infty}$ of $\ell_{p}$.  Note that  $y_{l}\stackrel{w}{\longrightarrow} 0$ and $S_{m}$ is a finite rank operator and  hence  it is compact operator.  First pick $l_{1}=1,m_{0}=0$, and then  sequentially choose a sequence $m_{1}, l_{2}, m_{2}, l_{3}, m_{3}, \dots $ (where $\{l_{k}\},\{m_{k}\}$  are both increasing sequences)  such that
$$
\|S_{m_{k-1}}y_{l_{k}}\|< \delta\beta^{k} \quad \text{ and } \quad \|y_{l_{k}}-S_{m_{k}}y_{l_{k}}\|<  \delta\beta^{k}
$$
For each that $k\in \mathbb{N}$,   let $z_{k}=S_{m_{k}}y_{l_{k}}-S_{m_{k-1}}y_{l_{k}}$.  Then $\{ z_{k}\}_{k=1}^{\infty}$  is block basic sequence of the basis $\{ e_{n}\}_{n=1}^{\infty}$, particularly $\{ z_{k}\}_{k=1}^{\infty}$  is basic sequence with basis constant less than or equal to $1$ (where $\sup_{m}\|S_{m}\| $ is called the \textit{basis constant} and  the basis constant of $\ell_{p}$ is 1 for $1\leq p< \infty$). Note that for each $k$, we have
$$
\|y_{l_{k}}-z_{k}\|< 2\delta\beta^{k}
\quad \text{ and } \quad
\|z_{k}\| > \delta-2\delta\beta=(1-2\beta)\delta.
$$
Therefore
$$
2\sum_{k=1}^{\infty}\dfrac{\|y_{l_{k}}-z_{k}\|}{\|z_{k}\|}\leq \sum_{k=1}^{\infty}\dfrac{4\delta\beta^{k}}{(1-2\beta)\delta}  =\dfrac{4\beta}{(1-2\beta)(1-\beta)} < 1
$$
By the Principle of small perturbations (\cite{AK},Theorem 1.3.9), $\{y_{l_{k}}\}_{k=1}^{\infty}$ is a basic sequence and equivalent to $\{z_{k}\}_{k=1}^{\infty}$.

 Recall that  a sequence $ \{ w_{n}\}_{n=1}^{\infty}$ is  called {\it seminormalized}  if  $$ 0< \inf_{n}\|w_{n}\| \leq \sup _{n}\|w_{n}\|<\infty, $$
and it is known that any seminormalized block basic sequence $ \{ w_{n}\}_{n=1}^{\infty}$ of $\{e_{n}\}_{n=1}^{\infty}$ is equivalent to $\{e_{n}\}_{n=1}^{\infty}$ in $\ell_{p}$  (see \cite{AK}).  Since for each $k$, we have
$$
\frac{5}{7}\delta  \leq \|z_{k}\| \leq  \|y_{l_{k}}\|+ 2\delta \beta \leq  \| 4\|U\|\|x_{i}\|+ \frac{2}{7}\delta,
$$
we get that $\{z_{k}\}_{k=1}^{\infty}$ is  seminormalized block basis of $\{e_{n}\}_{n=1}^{\infty}$ and  hence it is equivalent to $\{e_{k}\}_{k=1}^{\infty} $.
Thus $\{y_{l_{k}}\}_{k=1}^{\infty}$ is equivalent to $\{e_{k}\}_{k=1}^{\infty}$  by the definition of equivalence of basic sequences.

Finally as we have seen if (\ref{contradiction1}) holds, then we can choose some $(a_{l_{k}})_{k=1}^{\infty} \in \ell_{2}$ such that $\sum_{k=1}^{\infty}a_{l_{k}}y_{l_{k}}$ converges in norm by Claim 1.  By the  ``equivalence" stated in  Claim 2,  we get that $ \sum_{k=1}^{\infty}a_{l_{k}}e_{k}$  converges in $\ell_{p} (1\leq p < 2)$,   which contradicts with the fact that $\ell_{p}(1\leq p < 2)$ is a proper subspace of $\ell_{2}$. Therefore we must have
$$
\sup _{R^{\prime} \in B_{M}}\big\|\overline{U}\big(\sum_{j=M+1}^{\infty}\big(R^{\prime}P_{j}Q_{i}\big)\big) x_{i}\big\|_{\ell_{p}} \to 0,
$$
which implies that
$$
V(P)\Big(\sum_{i=1}^{N} C_{i} \overline{U}_{Q_{i}, x_{i}}\Big)=\sum_{j=1}^{\infty} V(P_{j})\Big(\sum_{i=1}^{N} C_{i} \overline{U}_{Q_{i}, x_{i}}\Big).
$$
Thus $V$ is  strongly countably  additive on $\mathcal{M}_{U, X}$.
\end{proof}


Recall that a {\it Jordan homomorphism} $\Phi: M \rightarrow N$ from  VN algebras $M$ to an algebra  $N$ is a linear map satisfying the condition $\Phi\left(a^{2}\right)=\Phi(a)^{2}$ for all $a \in M$ (equivalently, $\Phi(a b+b a)=\Phi(a) \Phi(b)+\Phi(b) \Phi(a)$ for all $a, b \in M)$.


\begin{theorem}\label{maptojordan}
Let $M$ be a $\sigma$-finite  VN algebra without Type $I_{2}$ direct summand, and X has Schur property or $X$ is $\ell_{p} (1\leq p < 2)$.   If $\varphi: M \rightarrow B(X)$ be a ultraweakly-wot continuous bounded linear  mapping, then there is a Banach space $Z$ such that $\varphi$ can be dilated to a ultraweakly-wot continuous Jordan homomorphism $\phi$ from $M$ to $B(Z)$.
\end{theorem}

\begin{proof}
Since $\varphi$ is a ultraweakly-wot  bounded linear mapping from $M$ to $B(X)$ and $M$ is $\sigma$-finite, by Proposition \ref{continuity},   it reduces  to a countably additive  operator-valued  quantum  measure $U: \mathcal{P}(M)\to B(X)$.  By  Theorem \ref{countablemeasuredilation}, there exists a dilation space $Z$ and bounded operators $S: Z\to X $ and $T: X \to Z $, a countably additive quantum measure $V: \mathcal{P}(M) \to B(Z)$ such that
$$
U(P)=SV(P)T
$$
and $V$ is a  projection-valued dilation of $U$.

By the generalized Gleason theorem (see Lemma \ref{Gleasonthm}), we can extend $V$ uniquely to a bounded linear mapping $\phi$ from $M$ to $B(Z).$  Furthermore, $\phi$ is countably additive on $\mathcal{P}(M)$. Applying Proposition \ref{continuity} again, we get that $\phi$ is ultraweakly-wot continuous. It remains to show that $\phi$ is a Jordan homomorphism.

Since $V$ is projection-valued,  we know that $\phi$ is idempotent on $\mathcal{P}(M)$ (i.e. $\phi(P)^2 = \phi(P)$ for any $P\in \mathcal{P}(M)$). Let $ P_{\alpha},P_{\beta} \in \mathcal{P}(M). $  If $ P_{\alpha} \perp P_{\beta} $, then $ P_{\alpha}+P_{\beta}= P_{\alpha}\vee P_{\beta} \in \mathcal{P}(M)$. Since $\phi$ is  idempotent on $\mathcal{P}(M)$,  we get  $\phi(P_{\alpha}\vee P_{\beta})=(\phi(P_{\alpha}\vee P_{\beta}))^{2}$.
By linearity, we have
$$
\phi(P_{\alpha})+\phi(P_{\beta})=\phi(P_{\alpha})^{2}
+\phi(P_{\beta})^{2}+\phi(P_{\alpha})\phi(P_{\beta})+\phi(P_{\beta})\phi(P_{\alpha}).
$$
Thus
$$
\phi(P_{\alpha})\phi(P_{\beta})+\phi(P_{\beta})\phi(P_{\alpha})=0.
$$

Now let $x=\sum_{i=1}^{N} \lambda_{i} P_{i} \in M$, where $P_{1}, \ldots, P_{n} \in \mathcal{P}(M)$ are mutually orthogonal and $\lambda_{i} \in \mathbb{R}.$  Then $x=x^*$.  Upon computing $\phi\left(x^{2}\right)$ we see that
$$
\phi(x^{2}) =\phi\Big(\sum_{i=1}^{n} \lambda_{i}^{2} P_{i}\Big)=\sum_{i=1}^{n} \lambda_{i}^{2} \phi\left(P_{i}\right)
$$
$$
\phi(x)^{2}=\Big(\sum_{i=1}^{N} \lambda_{i}\phi(P_{i})\Big)^{2}=\sum_{i=1}^{N} \lambda_{i}^{2}\phi(P_{i})+\sum_{1\leq i<j\leq N }\lambda_{i}\lambda_{j} \left(\phi(P_{i})\phi(P_{j})+\phi(P_{j})\phi(P_{i})\right)=\phi(x^{2})
$$

Since any self-adjoint element $a \in M_{sa}$ can be approximated by a finite real linear combination of mutually orthogonal  projection, then we get  $\phi(a^{2})=\phi(a)^{2}, \forall~ a\in M_{sa}$.

For any two projections $P_{1}, P_{2}\in \mathcal{P}(M)$. Then $y = P_1 + P_2\in M_{sa}$, and so  $\phi(y^{2})=\phi(y)^{2}$, which implies that
$$
\phi(P_{1})\phi(P_{2})+\phi(P_{2})\phi(P_{1})=\phi(P_{1}P_{2}+P_{2}P_{1}).
$$
Therefore, for any $z=\sum_{j=1}^{M}\kappa_{j}P_{j}$ with  $\{\kappa_{j}\}_{1 \leq j \leq M}\subset \mathbb{C}$ and $\{ P_{j}\}_{1 \leq j \leq M}\subset \mathcal{P}(M)$, we get
$$
\begin{aligned}
\phi(z^{2})&=\sum_{j=1}^{M}\kappa_{j}^{2}\phi(P_{j})+\sum_{1\leq i< j\leq M}\kappa_{i}\kappa_{j}\phi(P_{i}P_{j}+P_{j}P_{i})\\
&=\sum_{j=1}^{M}\kappa_{j}^{2}\phi(P_{j})^{2}+\sum_{1\leq i< j\leq M}\kappa_{i}\kappa_{j}\left(\phi(P_{1})\phi(P_{2})+\phi(P_{2})\phi(P_{1})\right)\\
&=\phi(z)^{2}.
\end{aligned}
$$
Since the set of finite linear combinations of projections in $M$ is norm dense in $M$, we obtain that $\phi$ is a Jordan homomorphism.

\end{proof}

%
%
%

\section{Quantum measures with bounded $p$-variation}



Let $\mathcal{M}(\Sigma, X)$ be the linear space  of vector measures on $\Sigma$ with values in $X$, where $\Sigma$ is a $\sigma$-algebra of subset of a set $\Omega$. Let $\mu $ be a vector measure, then the \textit{variation} of $\mu$ can be defined  as in the scalar case, even we can define the $p$-variation of $\mu$  by
$$
|\mu|_{p}(E)=\sup\Bigg\{ \Big(\sum_{i=1}^{n}\|\mu(E_{i})\|^{p}\Big)^{1/p} :  \{ E_{1}, \dots, E_{n}\} \text{ a disjoint partition of } E \Bigg\}.
$$
The vector measure $\mu$ on $\Sigma$ is said to have bounded $p$-variation (or bounded variation for $p=1$) if $|\mu|_{p}(E)$ is finite for every $E\in \Sigma, $ or equivalent, if $|\mu|_{p}(\Omega)$ is finite.

Now let $ \varphi \in \mathcal{M}(\Sigma, B(X))$ be an operator-valued measure. For each $x\in X$,  $\varphi_{x}: \Sigma\to X$ by $\varphi_{x}(E)=\varphi(E)x $ is a vector measure. Then the $p$-variation \cite{HLL} for the operator-valued measure $\varphi$ on the classical measure space $(\Omega, \Sigma)$ is defined by
$$
|\varphi|_{p}(E)=\sup_{\|x\|\leq 1}|\varphi_{x}|_{p}(E),~\forall ~ E\in \Sigma.
$$

To define a $p$-variation for quantum measures, we need to replace $\Sigma$ by $\mathcal{P}(M)$ and a partitions of $\Omega$ by an appropriate partitions of orthogonal projections via the following orthogonal representations of finite trees.


%
%
%

We start by reviewing the concept of the \textbf{tree}.  Taking  the  set of natural numbers  $\mathbb{N}$,  define  the set
$$
\bigcup_{n=1}^{\infty} \mathbb{N}^{n}=\left\{(a_{1},a_{2},\dots, a_{n} )|  n \in \mathbb{N}^+ \text{ and }  a_{k} \in \mathbb{N},   \forall ~1\leq k\leq n \right\}
$$
with  partially order  by taking $\left(a_{1}, \dots, a_{n}\right) \leq \left(a^{\prime}_{1}, \dots, a^{\prime}_{p}\right)$ provided $ p \geq n $ and $ a_{k}=a_{k}^{\prime} $ for $1\leq k \leq n $.  And a finite  tree $\mathrm{T}$ on $\mathbb{N}$ will be a finite subset of $\bigcup_{n=1}^{\infty} \mathbb{N}^{n}$ with the property that a predecessor of a member of $\mathrm{T}$  belongs also to $\mathrm{T}$, that is,  $\left(a_{1}, \dots, a_{n} \right) \in \mathrm{T} $ whenever $\left(a_{1}, \dots, a_{n} , a_{n+1}\right) \in \mathrm{T} $.

For a finite tree $\mathrm{T}$,  we call the elements of $\mathrm{T}$  the \textbf{nodes} of $\mathrm{T}$.
Let $t\in \mathrm{T}$. We say that $t$ is a \textbf{terminal} of $\mathrm{T}$ if $s=t$ whenever $s\in \mathrm{T}$ and $t\leq s$.
Moreover, we say $\mathrm{T}_{1}$ is \textbf{subtree} of  $\mathrm{T}$ if $\mathrm{T}_{1}$ is a tree and $\mathrm{T}_{1} \subset \mathrm{T}$.  We will use $\textbf{Ter}(\mathrm{T})$ to denote the set of all terminals of  $\mathrm{T}$ and  $\mathcal{T}$  denotes the set of all the finite trees.

A simple observation is that given a finite tree $\mathrm{T}$, then there exist a positive integer $L$ such that $\mathrm{T}\subset  \bigcup_{n=1}^{L}\mathbb{N}^{n}$ and $\mathrm{T}\bigcap \big(\bigcup_{n=1}^{L-M} \mathbb{N}^{n}\big)$ is subtree of $\mathrm{T}$ for all integer $1\leq M \leq L-1$. Besides, the  set of terminals  $\textbf{Ter}(\mathrm{T})$  of $\mathrm{T}$ determines the structure of the whole finite tree $\mathrm{T}$, as if a terminal $(a_{1}, a_{2} \dots a_{l})$ is fixed, then  we know that for each $ 1\leq k\leq l$,  $ (a_{1}, \dots , a_{k})$  is  the node of $\mathrm{T}$.  Of course, $\textbf{Ter}(\mathrm{T})$  is a finite set as well.

\begin{definition}\label{ORT}
Given a finite tree $\mathrm{T}$ with nodes $\{(a_{1}, a_{2},\dots, a_{k})\}$.  An \textit{orthogonal representation} of $\mathrm{T}$ (OR($\mathrm{T}$)) is a map
$$
\mathscr{P}:\mathrm{T}\to \mathcal{P}(M),\quad  (a_{1}, a_{2},\dots, a_{k}) \mapsto \mathscr{P}((a_{1}, a_{2},\dots, a_{k}))~(\text{denoted by } P_{(a_{1}, a_{2},\dots, a_{k})})
$$
such that  for any two distinct nodes $(a_{1}, a_{2},\dots, a_{k}^{i})$ and $(a_{1}, a_{2},\dots, a_{k}^{j})$ which differ only in the $k$-th coordinate,  we have
$$
P_{(a_{1}, a_{2},\dots, a_{k}^{i})  } \perp P_{(a_{1}, a_{2},\dots, a_{k}^{j})}.
$$
\end{definition}

Now we give the definition of the $p$-variation for operator-valued quantum measure.

\begin{definition}\label{p-variation}
Let $M$ be a VN algebra without type $I_{2}$ direct summand,  and  $U: \mathcal{P}(M) \rightarrow B(X)$ be a quantum measure. For every $P\in \mathcal{P}(M)$, the {\it $p$-variation } of $U$ on $P$ is defined by
\begin{align*}
|U|_{p}(P)
&=\sup_{\|x\|\leq 1}|U_{x}|_{p}(P)\\
& =\sup_{\|x\|\leq 1}  \sup_{\mathrm{T}\in \mathcal{T}, \mathscr{P}\in OR(\mathrm{T})}  \Bigg(\sum_{(a_{1},\dots, a_{l})\in \textbf{Ter}(\mathrm{T})}
\Big\|\overline{U}(P_{(a_{1},\dots, a_{l})}P_{(a_{1},\dots, a_{l-1})} \dots P_{(a_{1})} P)x \Big\|^{p} \Bigg)^{1/p},
\end{align*}
where $\overline{U}$ is the extension of $U$ and the ``sum" is over all the branches of projections like $P_{(a_{1},\dots, a_{l})}P_{(a_{1},\dots, a_{l-1})} \dots P_{(a_{1})}$ whose first coordinate indexed by  $\textbf{Ter}(\mathrm{T})$  (the set of terminals of finite tree $\mathrm{T}$) and the ``sup" is taken over all family of finite trees and their orthogonal representations.
\end{definition}

Obviously if $P_{1},  P_{2} \in \mathcal{P}(M)$ and $P_{1}\leq P_{2}$,  then  $|U|_{p}(P_{1})\leq |U|_{p}(P_{2})$. The operator-valued quantum measure $U$ is said to have bounded $p$-variation (or bounded variation for $p=1$) if $|U|_{p}(P)$ is finite for every $P\in \mathcal{P}(M)$, or equivalent, if $|U|_{p}(I)$ is finite, denoted by $\|U\|_{pV}$. By the Uniform Boundedness Principle, we know that $B(X)$-valued quantum measure $U$ has bounded $p$-variation if and only if for every $x\in X$, the vector (quantum) measure  $U_{x}$ has bounded $p$-variation.

If $M$ is an abelian VN algebra, then we can write $M=L^{\infty}(K, \nu)$,  where $(K, \nu)$ is locally compact space with Radon measure $\nu$. Thus $\mathcal{P}(M)$ can be identified with the structure of measurable subsets of $K$.  For any operator-valued quantum measure $U$, setting
$$
\mu(E)=U(P), ~\text{ if }~ P=\chi_{E}\in\mathcal{P}(M)
$$
defines an operator-valued measure on the $\sigma$-algebra of measurable subsets of $K$. Next we show that our definition of $p$-variation in this case agrees with the one defined in \cite{HLL}.

\begin{proposition} Let  $M=L^{\infty}(K, \nu)$ be an abelian VN algebra and $ U\in \mathcal{M}(\mathcal{P}(M), B(X))$ be a quantum measure. Then the $p$-variation of $U$  is the same as the $p$-variation of induced operator-valued measure $\mu \in \mathcal{M}(\Sigma, B(X))$, i.e., $|U|_{p}(P)=|\mu|_{p}(E)$ for any $P=\chi_{E}\in\mathcal{P}(M)$.
\end{proposition}

\begin{proof}  After identifying $P\in \mathcal{P}(M)$ with $\chi_{E} \ $, we get
$$
\begin{aligned}
|U|_{p}(P)
& =\sup_{\|x\|\leq 1}
\sup_{\mathrm{T}\in \mathcal{T},\mathscr{P}\in OR(\mathrm{T})} \Bigg(\sum_{(a_{1},\dots, a_{l})\in \textbf{Ter}(\mathrm{T})}
\left\|\overline{U}(P_{(a_{1},\dots, a_{l})}P_{(a_{1},\dots, a_{l-1})} \dots P_{(a_{1})} P)x \right\|^{p} \Bigg)^{1/p}\\
& =\sup_{\|x\|\leq 1}
\sup_{\mathrm{T}\in \mathcal{T},\mathscr{P}\in OR(\mathrm{T})} \Bigg(\sum_{(a_{1},\dots, a_{l})\in \textbf{Ter}(\mathrm{T})}
\left\|\overline{U}(\chi_{ E_{(a_{1},\dots, a_{l})} }\chi_{E_{(a_{1},\dots, a_{l-1})}} \dots \chi_{E_{(a_{1})}} \chi_{E})x \right\|^{p} \Bigg)^{1/p}\\
&=\sup_{\|x\|\leq 1}
\sup_{\mathrm{T}\in \mathcal{T},\mathscr{P}\in OR(\mathrm{T})} \Bigg(\sum_{(a_{1},\dots, a_{l})\in \textbf{Ter}(\mathrm{T})}
\left\|\overline{U}(\chi_{E_{(a_{1},\dots, a_{l})} \cap  E_{(a_{1},\dots, a_{l-1})} \cap  \dots \cap E_{(a_{1})}\cap E } )x \right\|^{p} \Bigg)^{1/p}\\
\end{aligned}
$$

Note that for a given finite tree $\mathrm{T}$ and  any two of  its terminals:  $(a_{1}, \dots, a_{l-1},  a_{l})$ and $(a_{1}^{\prime}, \dots, a_{m-1}^{\prime},  a_{m}^{\prime})$, we set
$$
F_{(a_{1}, \dots, a_{l-1},  a_{l})}=E_{(a_{1},\dots,a_{l-1},  a_{l})} \cap  E_{(a_{1},\dots, a_{l-1})} \cap  \dots \cap E_{(a_{1})}\cap E
$$
and
$$
F_{(a_{1}^{\prime}, \dots, a_{m-1}^{\prime},  a_{m}^{\prime})}=E_{(a_{1}^{\prime}, \dots, a_{m-1}^{\prime},  a_{m}^{\prime})} \cap  E_{(a_{1}^{\prime}, \dots, a_{m-1}^{\prime})} \cap  \dots \cap E_{(a_{1}^{\prime})}\cap E.
$$
as  two terminals are different, then they will split at $r$-th coordinate for some $1 \leq r\leq \min\{l, m\}$, to be more specific, $r$ is the smallest integer such that $ a_{r}\neq a_{r}^{\prime}$, then by the definition of orthogonal representation, we have $P_{(a_{1},a_{2}, \dots a_{r})} \perp P_{(a_{1}^{\prime}, a_{2}^{\prime}, \dots a_{r}^{\prime})}$, that is, $E_{(a^{1}, a_{2},\dots, a^{r})} \cap E_{(a_{1}^{\prime}, a_{2}^{\prime}, \dots,  a_{r}^{\prime})}=\emptyset$, meanwhile
$$
F_{(a_{1}, \dots, a_{l-1},  a_{l})}\subset E_{(a_{1}, a_{2},\dots, a^{r})}
\quad
F_{(a_{1}^{\prime}, \dots, a_{m-1}^{\prime},  a_{m}^{\prime})} \subset  E_{(a_{1}^{\prime}, a_{2}^{\prime}, \dots,  a_{r}^{\prime})}
$$
thus  $\{ F_{(a_{1},a_{2},\dots, a_{l})} \}_{(a_{1},a_{2},\dots, a_{l})\in \textbf{Ter}(\mathrm{T} ) }$ is mutually  disjoint subset of $E$.
While for each  finite tree $\mathrm{T}$ and one of its orthogonal representations corresponds to a partition of $E$.  For simplicity, we write $(a_{1}, \dots, a_{l-1},  a_{l})=\alpha$ and $\Lambda=\textbf{Ter}(\mathrm{T})$, then
$$
\begin{aligned}
|U|_{p}(P)
&=\sup_{\|x\|\leq 1}\sup_{\mathrm{T}\in \mathcal{T},\mathscr{P}\in OR(\mathrm{T})}
\Bigg(\sum_{(a_{1},\dots, a_{l})\in \textbf{Ter}(\mathrm{T})} \Big\|\overline{U}(\chi_{F_{(a_{1},a_{2},\dots, a_{l})}})x \Big\|^{p} \Bigg)^{1/p}\\
&=\sup_{\|x\|\leq 1} \sup_{\mathrm{T}\in \mathcal{T},\mathscr{P}\in OR(\mathrm{T})}
\Bigg(\sum_{(a_{1},\dots, a_{l})\in \textbf{Ter}(\mathrm{T})}
\Big\|\mu(F_{(a_{1},a_{2},\dots, a_{l})})x \Big\|^{p} \Bigg)^{1/p}\\
&=\sup_{\|x\|\leq 1} \sup \Bigg\{  \Big( \sum_{\alpha \in \Lambda } \|\mu(F_{\alpha})x\|^{p} \Big)^{1/p} : \text{  $\{ F_{\alpha}\}_{\alpha \in \Lambda}$ is a finite disjoint partition of $E$}  \Bigg\}\\
&=|\mu|_{p}(E)
\end{aligned}
$$
\end{proof}

Next we present a couple of examples with bounded $p$-variations.

\begin{example}
Recall  that an operator space is a closed space of $B(H)$  \cite{Pa,Pi}. Let  $\phi$ be a linear map between operator spaces $\mathcal{A} $ and $\mathcal{B}$.  For any $n\geq 1$, define  $\phi_{n}: M_{n}(\mathcal{A}) \to M_{n}(\mathcal{B})$ by
$$
\phi_{n}\big((a_{i,j}) \big)=\big(\phi(a_{i,j}) \big), \text{ where } a_{i,j}\in \mathcal{A}.
$$
The map $\phi$ is called completely bounded (c.b. in short) if
$
\sup_{n\geq 1}\|\phi_{n} \|_{M_{n}(\mathcal{A}) \to M_{n}(\mathcal{B})} < \infty
$
and its completely bounded norm is defined by $\| \phi \|_{cb}=\sup_{n\geq 1}\|\phi_{n}\|_{M_{n}(\mathcal{A}) \to M_{n}(\mathcal{B})}.$   We say that $\phi$  is completely positive (c.p. in short) if $\phi_{n}((a_{i,j}))\geq 0$  whenever $(a_{i,j})\geq 0$ in $ M_{n}(\mathcal{A}) $
for any $ n\geq 1$.

Let $K, H$ be Hilbert spaces and $M\subset B(K)$  be a VN algebra. Suppose $\Psi : M \to B(H)$ is a normal completely bounded map and let $\psi=\Psi|_{\mathcal{P}(M)}$. Then $\psi$ is an  operator-valued quantum measure with bounded  $2$-variation.
\end{example}

\begin{proof} By Stinespring's dilation theorem and the factorization property for c.b. maps \cite{Pa, Pi, Li},  we know that there exist a Hilbert space $\widehat{H}$, a normal $*$-homomorphism $\pi: B(K)\to B(\widehat{H})$ and linear operators $V_{1}: H\to \widehat{H}, V_{2}:  H \to \widehat{H} $ with $\|V_{1}\|\|V_{2}\|=\|\Psi\|_{cb}$  such that $\Psi(u)=V_{2}^{*}\pi(u)V_{1}$ for all $ u\in M $. Moreover, by the definition, we have
$$
\begin{aligned}
|\psi|_{2}(P)
&=\sup_{\|x\|_{H}\leq 1}\sup_{\mathrm{T}\in \mathcal{T},\mathscr{P}\in OR(\mathrm{T})}\Bigg(\sum_{(a_{1},\dots, a_{l})\in \textbf{Ter}(\mathrm{T})}\left\|\Psi(P_{(a_{1},\dots, a_{l})}P_{(a_{1},\dots, a_{l-1})} \dots P_{(a_{1})} P)x \right\|^{2} \Bigg)^{1/2}\\
&=\sup_{\|x\|_{H}\leq 1}\sup_{\mathrm{T}\in \mathcal{T},\mathscr{P}\in OR(\mathrm{T})}\Bigg(\sum_{(a_{1},\dots, a_{l})\in \textbf{Ter}(\mathrm{T})}\left\|V_{2}^{*}\pi(P_{(a_{1},\dots, a_{l})}P_{(a_{1},\dots, a_{l-1})} \dots P_{(a_{1})} P)V_{1}x \right\|^{2} \Bigg)^{1/2}\\
&\leq \|V_{2}^{*}\| \sup_{\|x\|_{H}\leq 1}\sup_{\mathrm{T}\in \mathcal{T},\mathscr{P}\in OR(\mathrm{T})}\Bigg(\sum_{(a_{1},\dots, a_{l})\in \textbf{Ter}(\mathrm{T})}\left\|\pi(P_{(a_{1},\dots, a_{l})}) \dots \pi(P_{(a_{1})}) \pi( P)V_{1}x \right\|^{2} \Bigg)^{1/2}\\
\end{aligned}
$$

Let $\mathrm{T}$ be any given a finite tree which is subset of $\bigcup_{n=1}^{L}\mathbb{N}^{n}$ and $\mathscr{P}$ be its orthogonal representation. For non-terminal nodes $(a_{1}, \dots, a_{r})$ and let $E_{(a_{1}, \dots, a_{r})}$  be the set of all nodes whose predecessor is  $(a_{1}, \dots, a_{r})$. If the nodes  $(a_{1}, \dots, a_{r}, a_{r+1}^{i}),$  $(a_{1}, \dots, a_{r}, a_{r+1}^{j})$ are  in  $E_{(a_{1}, \dots, a_{r})}$,  by the definition of orthogonal representation of $\mathrm{T}$ (OR($\mathrm{T}$)), then $P_{(a_{1}, \dots, a_{r}, a_{r+1}^{i})} \perp P_{(a_{1}, \dots, a_{r}, a_{r+1}^{j})}$. Furthermore,  since $\pi$ is a $*$-homomorphism, we have that for every $P\in \mathcal{P}(M)$, $\pi(P)$ is also a projection, and $\pi(P_{1})\perp \pi(P_{2})$ whenever $P_{1} \perp P_{2}$. Thus the elements of the set
$$
\{ \pi( P_{(a_{1}, \dots, a_{r}, a_{r+1}^{i})} ) \}_{(a_{1}, \dots, a_{r}, a_{r+1}^{i})\in E_{(a_{1}, \dots, a_{r})}}
$$
are mutually orthogonal projections in $B(\widehat{H})$, and so
$$
\begin{aligned}
 \sum_{(a_{1}, \dots, a_{r}, a_{r+1}^{i})\in E_{(a_{1}, \dots, a_{r})} }
&\left\|\pi(P_{(a_{1}, \dots, a_{r}, a_{r+1}^{i})}) \pi(P_{(a_{1}, \dots, a_{r})}) \dots \pi( P_{(a_{1})} )\pi(P)V_{1} x \right\|^{2} \\
\leq & \left\| \pi(P_{(a_{1}, \dots, a_{r})}) \dots \pi( P_{(a_{1})}) \pi(P)V_{1}x \right\|^{2}.
\end{aligned}
$$
Therefore
$$
\begin{aligned}
& \quad  \quad \sum_{(a_{1},\dots, a_{l})\in \textbf{Ter}(\mathrm{T})}  &
& \left\|  \pi(P_{(a_{1},\dots, a_{l})}) \pi(P_{(a_{1},\dots, a_{l-1})}) \dots \pi(P_{(a_{1})}) \pi(P) V_{1}x \right\|^{2}  \\
\leq & \sum_{(a_{1},\dots , a_{m}) \in \textbf{Ter}(\mathrm{T}\cap(\bigcup_{n=1}^{L-1}\mathbb{N}^{n}))}  &
& \left\|  \pi(P_{(a_{1},\dots , a_{m})}) \pi(P_{(a_{1},\dots, a_{m-1})}) \dots \pi(P_{(a_{1})}) \pi(P) V_{1}x \right\|^{2} \\
\leq &  \sum_{(a_{1},\dots , a_{r}) \in  \textbf{Ter}(\mathrm{T}\cap(\bigcup_{n=1}^{L-2}\mathbb{N}^{n}))}   &
&\left\| \pi(P_{(a_{1},\dots, a_{r})}) \pi(P_{(a_{1},\dots, a_{r-1})}) \dots \pi(P_{(a_{1})}) \pi(P) V_{1}x \right\|^{2}\\
& \quad \dots \dots  &  \\
\leq & \qquad  \sum_{(a_{1}) \in \textbf{Ter}(\mathrm{T}\cap \mathbb{N} )} &
&\left\| \pi(P_{(a_{1})}) \pi( P )V_{1}x \right\|^{2} \\
\leq {} & ~\|\pi(P)V_{1}x\|^{2}   \\
\leq {} & ~ \|V_{1}x\|^{2}.
\end{aligned}
$$
where  $\mathrm{T}\cap (\bigcup_{n=1}^{L-M}\mathbb{N}^{n})$ is a subtree of $\mathrm{T}$, and the sum is taken over all the terminals $\textbf{Ter}(\mathrm{T}\cap(\bigcup_{n=1}^{L-M}\mathbb{N}^{n}))$ of finite tree $\mathrm{T}\cap (\cup_{n=1}^{L-M}\mathbb{N}^{n})$ for $1 \leq M \leq L-1$. Finally, we obtain
$$
|\psi|_{2}(P)\leq \sup_{\|x\|_{H}\leq 1} \|V_{2}^{*}\|\|V_{1}x\|= \|V_{2}^{*}\|\|V_{1}\|=\|\Psi\|_{cb}, ~ \forall ~ P\in \mathcal{P}(M),
$$
and hence $\psi$ has bounded  $2$-variation.
\end{proof}

In addition, there is a concrete example in non-commutative $ L_{p}$-spaces \cite{PX}.

\begin{example} Let $M$ be a VN algebra, and assume that $\tau$ is a normal semi-finite faithful trace on $M$. Then the pair $(M, \tau)$ is called a \textit{noncommutative measure space }. Set $ S_{+}(M)=\{x\in M_{+}: \tau(s(x))< \infty\}, $ where the support $s(x)$ is defined as the minimal projection such that $s(x)xs(x)=x$. Let $ S(M)=\operatorname{span}\{S_{+}(M)\}$ and define $\|x\|_{p}=\tau(|x|^{p})^{1/p}, 1\leq  p < \infty$. Then the noncommutative $L^{p}$ space is defined as
$$
L^{p}(M, \tau)=\overline{S(M)^{\|\cdot\|_{p}}}
$$
 Now assume  that $M$ is a VN algebra with  normal, finite, faithful trace $\tau$,  and let the linear mapping $\Phi: M \to B(L^{p}(M, \tau))$  be given by $\Phi(x)=L_{x}, \forall x\in M$, where $L_{x}(y)=xy, ~\forall~ y\in L^{p}(M, \tau)$. Then its restriction on $\mathcal{P}(M)$, denoted by $\phi$,   is a projection-valued quantum measure with bounded $p$-variation for every $p\geq 2$.
\end{example}

\begin{proof} As non-commutative H\"{o}der inequality implies  that for any $x\in S(M),$ and $ a, b\in  M$, we have $\|axb\|_{p}\leq \|a\|\|x\|_{p}\|b\|$ for  $1 \leq p < \infty$ and thus $\|L_{x}\|\leq \|x\|$.  Then
clearly,  $\Phi$ is a bounded  unital homomorphism, and so its restriction on $\mathcal{P}(M)$ is a projection-valued quantum measure.

Moreover, it is true that if $p\geq 2$ and  $P_{1}, P_{2},\dots, P_{n}$ are mutually orthogonal projections in $S(M)$, then
\begin{equation}\label{inequ 1}
\Big(\sum_{i=1}^{n}\|P_{i}y\|_{p}^{p} \Big)^{1/p}\leq \|y\|_{p}, \quad y\in S(M)
\end{equation}
By the definition, we have
$$
\begin{aligned}
|\phi_{x}|_{p}(P)
&=\sup_{\mathrm{T}\in \mathcal{T},\mathscr{P}\in OR(\mathrm{T})}\Bigg(\sum_{(a_{1},\dots, a_{l})\in \textbf{Ter}(\mathrm{T})}
\Big\|\Phi(P_{(a_{1},\dots, a_{l})} P_{(a_{1},\dots, a_{l-1})} \dots P_{(a^{1}_{\alpha})} P)x \Big\|^{p} \Bigg)^{1/p}\\
&=\sup_{\mathrm{T}\in \mathcal{T},\mathscr{P}\in OR(\mathrm{T})}\Bigg(\sum_{(a_{1},\dots, a_{l})\in \textbf{Ter}(\mathrm{T})}
\Big\|P_{(a_{1},\dots, a_{l})} P_{(a_{1},\dots, a_{l-1})} \dots P_{(a^{1}_{\alpha})} Px \Big\|^{p} \Bigg)^{1/p}.\\
\end{aligned}
$$
Given a  finite tree $\mathrm{T} (\subset \bigcup_{n=1}^{L}\mathbb{N}^{n})$. By a similar argument, for some  non-terminal node $(a_{1}, \dots, a_{r})$, let the set $E_{(a_{1}, \dots, a_{r})}$ denote all the nodes whose predecessor is  $(a_{1}, \dots, a_{r})$. By the definition of OR($\mathrm{T}$), for any two distinct nodes $(a_{1}, \dots, a_{r}, a_{r+1}^{i}), (a_{1}, \dots, a_{r}, a_{r+1}^{j}) $ in $ E_{(a_{1}, \dots, a_{r})}$, we have
$$
P_{(a_{1}, \dots, a_{r}, a_{r+1}^{i})} \perp P_{(a_{1}, \dots, a_{r}, a_{r+1}^{j})}.
$$
that is, $\{ P_{(a_{1}, \dots, a_{r}, a_{r+1}^{i})} \}$ is the set of mutual orthogonal projections. By  applying the above inequality (\ref{inequ 1}), we have
$$
\begin{aligned}
\sum_{(a_{1}, \dots, a_{r}, a_{r+1}^{i})\in E_{(a_{1}, \dots, a_{r})}}
{} &  \left\|P_{(a_{1}, \dots, a_{r}, a_{r+1}^{i})} P_{(a_{1}, \dots, a_{r})} \dots P_{(a_{1})}Px \right\|^{p} .\\
\leq {} &  \left\| P_{(a_{1}, \dots, a_{r})} \dots P_{(a_{1})}Px \right\|^{p} .\\
\end{aligned}
$$


Thus we have
$$
\begin{aligned}
& \quad  \quad \sum_{(a_{1},\dots, a_{l})\in \textbf{Ter}(\mathrm{T})}  &
& \left\|  P_{(a_{1},\dots,  a_{l})} P_{(a_{1},\dots, a_{l-1})} \dots P_{(a_{1})} P x \right\|^{p}  \\
\leq & \sum_{(a_{1},\dots , a_{m}) \in \textbf{Ter}(\mathrm{T}\cap(\bigcup_{n=1}^{L-1}\mathbb{N}^{n}))}  &
& \left\| P_{(a_{1},\dots,  a_{m})} P_{(a_{1},\dots, a_{m-1})} \dots P_{(a_{1})} P x \right\|^{p} \\
\leq &  \sum_{(a_{1},\dots , a_{r}) \in  \textbf{Ter}(\mathrm{T}\cap(\bigcup_{n=1}^{L-2}\mathbb{N}^{n}))}   &
&\left\| P_{(a_{1},\dots,  a_{r})} P_{(a_{1},\dots, a_{r-1})} \dots P_{(a_{1})} P x \right\|^{p}\\
& \quad \dots \dots  &  \\
\leq & \qquad  \sum_{(a_{1}) \in \textbf{Ter}(\mathrm{T}\cap \mathbb{N})} &
&\left\| P_{(a_{1})} P x \right\|^{p} \\
\leq  & ~\|Px\|^{p}   \\
\leq  & ~ \|x\|^{p}.
\end{aligned}
$$
where for any $1 \leq M \leq L-1$,  $\textbf{Ter}(\mathrm{T}\cap(\bigcup_{n=1}^{L-M}\mathbb{N}^{n}))$ denotes all the terminals of subtree $\mathrm{T}\cap(\bigcup_{n=1}^{L-M}\mathbb{N}^{n})$.
Therefore, we get
$
|\phi_{x}|_{p}(P) \leq \|x\|_{p} ~ (\forall ~ P\in \mathcal{P}(M)),
$
which  implies that  $\phi$ is a projection-valued quantum measure with bounded $p$-variation with  $\|\phi\|_{pV}\leq 1$.
\end{proof}

\begin{proposition} Let $M$ be a VN algebra without type $I_{2}$ direct summand and  $V: \mathcal{P}(M) \rightarrow B(Y)$ be an operator-valued quantum measure  with bounded $p$-variation. Assume that  $S: Y \rightarrow X$ and $T: X \rightarrow Y$ are bounded linear maps. Then the compressed quantum measure $U(\cdot)=S V(\cdot) T: \mathcal{P}(M) \rightarrow B(X)$ has bounded $p$-variation and moreover
$$
|U|_{p}(P) \leq\|S\||V|_{p}(P)\|T\|
$$
for any $P \in \mathcal{P}(M)$.
\end{proposition}

\begin{proof}  Given any $P \in \mathcal{P}(M)$. Let $\overline{U},\overline{V}$ be the extensions of $U$ and $V$ respectively.  Clearly we have $\overline{U}=S \overline{V}(\cdot) T$.  By the definition, we have
$$
\begin{aligned}
|U|_{p}(P)&=\sup_{\|x\|\leq 1}|U_{x}|_{p}(P)\\
& =\sup_{\|x\|\leq 1} \sup_{\mathrm{T}\in \mathcal{T},\mathscr{P}\in OR(\mathrm{T})} \Bigg(\sum_{(a_{1},\dots, a_{l})\in \textbf{Ter}(\mathrm{T})}\left\|\overline{U}(P_{(a_{1},\dots, a_{l})}P_{(a_{1},\dots, a_{l-1})} \dots P_{(a_{1})} P)x \right\|^{p} \Bigg)^{1/p}\\
& =\sup_{\|x\|\leq 1} \sup_{\mathrm{T}\in \mathcal{T},\mathscr{P}\in OR(\mathrm{T})} \Bigg(\sum_{(a_{1},\dots, a_{l})\in \textbf{Ter}(\mathrm{T})} \left\|S\overline{V}(P_{(a_{1},\dots, a_{l})}P_{(a_{1},\dots, a_{l-1})} \dots P_{(a_{1})} P)Tx \right\|^{p} \Bigg)^{1/p}\\
&\leq \|S\|\cdot \sup_{\|x\|\leq 1} \sup_{\mathrm{T}\in \mathcal{T},\mathscr{P}\in OR(\mathrm{T})} \Bigg(\sum_{(a_{1},\dots, a_{l})\in \textbf{Ter}(\mathrm{T})} \left\|\overline{V}(P_{(a_{1},\dots, a_{l})}P_{(a_{1},\dots, a_{l-1})} \dots P_{(a_{1})} P)Tx \right\|^{p} \Bigg)^{1/p}\\
&\leq \|S\|\cdot \Big( \sup_{\|x\|\leq 1}  |V|_{p}(P)\|Tx\|\Big)\\
&= \|S\|  |V|_{p}(P)\|T\|.\\
\end{aligned}
$$
\end{proof}

Now we prove our main result of this section.

\begin{theorem} \label{dilation-mai} Let $M$ be a VN algebra without type $I_{2}$ direct summand and  $1\leq p <  \infty$. Then every  operator-valued quantum measure $U$ with bounded $p$-variation has a dilation to a  projection-valued quantum  measure $V$ with contractive $p$-variation.
\end{theorem}



\begin{proof}  We first define a $p$-variation norm on the elementary dilation space. Let $\overline{U}$ be the bounded linear extension of $U$ and let $\Phi= \sum_{i=1}^{N}C_{i}\overline{U}_{Q_{i},x_{i}}\in \mathcal{M}_{U} $. Define $\|\cdot\|_{pV}: \mathcal{M}_{U} \to \mathbb{R}^{+}\cup 0$ by
\begin{align*}
\|\Phi\|_{pV}
& =\sup_{\mathrm{T}\in \mathcal{T},\mathscr{P}\in OR(\mathrm{T})}\Bigg(\sum_{(a_{1}, \dots, a_{l})\in \textbf{Ter}(\mathrm{T})}\left\|\Phi(P_{(a_{1}, \dots,  a_{l})}P_{(a_{1}, \dots, a_{l-1})}\dots P_{(a_{1})})\right\|_{X}^{p} \Bigg)^{1/p}\\
& =\sup_{\mathrm{T}\in \mathcal{T},\mathscr{P}\in OR(\mathrm{T})}\Bigg( \sum_{(a_{1}, \dots, a_{l})\in \textbf{Ter}(\mathrm{T})}\Big\|\sum_{i=1}^{N}C_{i}\overline{U}(P_{(a_{1}, \dots,  a_{l})}P_{(a_{1}, \dots, a_{l-1})}\dots P_{(a_{1})} Q_{i})x_{i}\Big\|_{X}^{p} \Bigg)^{1/p}\\
\end{align*}
the ``sum" is over the  set of  terminals  $\textbf{Ter}(\mathrm{T})$ of finite tree $\mathrm{T}$ and the ``sup" is taken over all family of finite trees $\mathrm{T} \in \mathcal{T}$ and their orthogonal representations $\mathscr{P}\in OR(\mathrm{T})$.

By definition,  the value $\|\Phi\|_{pV}$  is independent of the representations of the operator $\Phi$.  First note that
$$
\begin{aligned}
\|\Phi\|_{pV}
&\leq \sum_{i=1}^{N}|C_{i}|\sup_{\mathrm{T}\in \mathcal{T},\mathscr{P}\in OR(\mathrm{T})}\Bigg( \sum_{(a_{1}, \dots, a_{l})\in \textbf{Ter}(\mathrm{T})}\Big\|\overline{U}(P_{(a_{1}, \dots,  a_{l})}P_{(a_{1}, \dots, a_{l-1})}\dots P_{(a_{1})} Q_{i})x_{i}\Big\|_{X}^{p} \Bigg)^{1/p}.
\end{aligned}
$$

Recall that every norm-one element can be  expressed  as a linear combination of four positive elements with the norm at most one, and moreover, if $x\in M $ with $0\leq x\leq I$, then there is a sequence of projections  $\{P_n\}$ such that $x=\sum_{n=1}^{\infty}P_{n}/2^{n}$ \cite{Mu}. So we can write  $Q_{i}=\sum_{l=1}^{4}\textrm{i}^{l}Q_{i,l}$, where \textrm{i} denotes the complex unit and $Q_{i, l}=\sum_{k=1}^{\infty}P_{i,l,k}/2^{k}$. In what follows we will also use following Minkowski inequality: for any $a_{kj} \geq 0$,

\begin{equation}\label{inequ 2}
\Bigg(\sum_{j=1}^{m}\Big(\sum_{k=1}^{\infty} \alpha_{k j}\Big)^{p}\Bigg)^{1/p} \leq \sum_{k=1}^{\infty}\left(\sum_{j=1}^{m} \alpha_{k j}^{p}\right)^{1/p}.
\end{equation}

Now we get
$$
\begin{aligned}
&\|\Phi\|_{pV}\\
&\leq \sum_{i=1}^{N}|C_{i}|\cdot  4 \max_{1\leq l\leq 4} \Bigg( \sup_{\mathrm{T}\in \mathcal{T},\mathscr{P}\in OR(\mathrm{T})}\Bigg( \sum_{(a_{1},\dots, a_{l})\in \textbf{Ter}(\mathrm{T})}\Big\|\overline{U}(P_{(a_{1}, \dots,  a_{l})}\dots P_{(a_{1})} Q_{i,l})x_{i}\Big\|_{X}^{p} \Bigg)^{1/p} \Bigg)\\
&= \sum_{i=1}^{N}|C_{i}| \cdot 4  \max_{1\leq l\leq 4} \Bigg(\sup_{\mathrm{T}\in \mathcal{T},\mathscr{P}\in OR(\mathrm{T})}\Bigg( \sum_{(a_{1}, \dots, a_{l})\in \textbf{Ter}(\mathrm{T})}\Big\|\overline{U}\Big(P_{(a_{1}, \dots,  a_{l})}\dots P_{(a_{1})} \big(\sum_{k=1}^{\infty}\dfrac{1}{2^{k}}P_{i,l,k}\big) \Big) x_{i}\Big\|_{X}^{p}\Bigg)^{1/p} \Bigg)\\
&\leq \sum_{i=1}^{N}|C_{i}| \cdot 4  \max_{1\leq l\leq 4} \Bigg( \sup_{\mathrm{T}\in \mathcal{T},\mathscr{P}\in OR(\mathrm{T})}\Bigg( \sum_{(a_{1}, \dots, a_{l})\in \textbf{Ter}(\mathrm{T})}\Big(\sum_{k=1}^{\infty}\dfrac{1}{2^{k}}\|\overline{U}(P_{(a_{1}, \dots,  a_{l})}\dots P_{(a_{1})} P_{i,l,k})x_{i}\|_{X}\Big)^{p} \Bigg)^{1/p} \Bigg)\\
&\leq \sum_{i=1}^{N}|C_{i}| \cdot 4  \max_{1\leq l\leq 4}
\Bigg(\sup_{\mathrm{T}\in \mathcal{T},\mathscr{P}\in OR(\mathrm{T})}
\Bigg(\sum_{k=1}^{\infty}\dfrac{1}{2^{k}} \Big( \sum_{(a_{1}, \dots, a_{l})\in \textbf{Ter}(\mathrm{T})}\|\overline{U}(P_{(a_{1}, \dots,  a_{l})}\dots P_{(a_{1})} P_{i,l,k})x_{i}\|_{X}^{p} \Big)^{1/p}  \Bigg) \Bigg)\\
&\leq \sum_{i=1}^{N}|C_{i}| \cdot 4 \max_{1\leq l\leq 4} \Big( \sum_{k=1}^{\infty} \dfrac{1}{2^{k}} \|U\|_{pV}\|x_{i}\| \Big)\\
&\leq 4\|U\|_{pV}\Big( \sum_{i=1}^{N}|C_{i}|\|x_{i}\|\Big)\\
\end{aligned}
$$
Thus $\|\cdot\|_{pV}$ is finite and  well-defined.
Let  $\Psi \in \mathcal{M}_{U}$, if  $\|\Psi\|_{pV}=0$, then by  taking $\mathrm{T}=\{(0)\}$ we get  $ \sup_{ P_{(0)}\in \mathcal{P}(M)} \|\Psi(P_{(0)})\|_{X}=0$ and hence $\Psi(P)=0$ for any $P\in \mathcal{P}(M)$. By the linearity of $\Psi$ and the set of finite linear combinations of projections in $M$ is dense in $M$, we have $\Psi=0$.  The homogeneity and triangle inequality can be easily verified by some simple computations. Thus  $\|\cdot\|_{pV}$ defines a norm on $\mathcal{M}_{U}$, and we  will denote its completion by $\widetilde{\mathcal{M}}_{U,pV}$.

Recall that the linear map $S: \widetilde{\mathcal{M}}_{U, pV}\to X,$  $T:  X \to \widetilde{\mathcal{M}}_{U, pV}$ are well-defined by Equations (\ref{equS}), (\ref{equT}). Furthermore,
$$
\Big\|S\Big(\sum_{i=1}^{N}C_{i}\overline{U}_{Q_{i},x_{i}}\Big)\Big\|
=\Big\|\sum_{i=1}^{N}C_{i}\overline{U}(Q_{i})x_{i}\Big\|
\leq \Big\|\sum_{i=1}^{N}C_{i}\overline{U}_{Q_{i},x_{i}}\Big\|_{pV}
$$
and
$$
\|Tx\|=\sup_{\mathrm{T}\in \mathcal{T},\mathscr{P}\in OR(\mathrm{T})}\Bigg( \sum_{(a_{1}, \dots, a_{l})\in \textbf{Ter}(\mathrm{T})}\Big\|\overline{U}(P_{(a_{1}, \dots,  a_{l})}P_{(a_{1}, \dots, a_{l-1})}\dots P_{(a_{1})} )x\Big\|_{X}^{p}\Bigg)^{1/p}
\leq \|U\|_{pV}\|x\|.
$$
Hence $S$ and $T$ are bounded linear operators with $ \|S\|\leq 1$ and $\|T\|\leq \|U\|_{pV}.$

The mapping  $V: \mathcal{P}(M) \to B(\widetilde{\mathcal{M}}_{U, pV})$ defined by  Equation (\ref{equV}) is finitely additive and idempotent valued, and  for every $P$ in $\mathcal{P}(M)$, we have
$$
\begin{aligned}
&\Big\|V(P)\Big(\sum_{i=1}^{N}C_{i}\overline{U}_{Q_{i},x_{i}}\Big)\Big\|=\Big\|\sum_{i=1}^{N}C_{i}\overline{U}_{PQ_{i},x_{i}} \Big\|_{pV}\\
& =\sup_{\mathrm{T}\in \mathcal{T},\mathscr{P}\in OR(\mathrm{T})}\Bigg( \sum_{(a_{1}, \dots, a_{l})\in \textbf{Ter}(\mathrm{T})}\Big\|\sum_{i=1}^{N}C_{i}\overline{U}(P_{(a_{1}, \dots,  a_{l})}P_{(a_{1}, \dots, a_{l-1})}\dots P_{(a_{1})} PQ_{i})x_{i}\Big\|_{X}^{p} \Bigg)^{1/p}\\
&\leq \Big\|\sum_{i=1}^{N}C_{i}\overline{U}_{Q_{i},x_{i}}\Big\|_{pV}\\
\end{aligned}
$$
Thus $V$ is a  projection-valued quantum measure.  By Generalized Gleason Theroem, we extend $V$ uniquely to  a bounded map $\overline{V}$ form $M$ to  $B(\widetilde{\mathcal{M}}_{U, pV})$

Finally, we prove that $V$ has bounded  $p$-variation. Let $y\in \widetilde{\mathcal{M}}_{U, pV}$, then there exists a sequence of  $y_{m}\in \mathcal{M}_{U}$ converging  to $y$.
Let $y_{m}$ be represented as
$
y_{m}=\sum_{i=1}^{N_{m}}C_{i}^{m}\overline{U}_{Q_{i}^{m}, x^{m}_{i}}
$
where $N_{m}$ is positive integer related to $m$ and  $\{ C_{i}^{m} \}\subset \mathbb{C}, \{ Q_{i}^{m} \} \subset B_{M}, \{x^{m}_{i}\}\subset X.$

Given a finite tree $\mathrm{T}_{1} $ with nodes $(a_{1}, a_{2}, \dots, a_{r})$ and let $\mathscr{P}_{1}$  be an orthogonal representation  of $\mathrm{T}_{1}$. Setting $\mathscr{P}_{1}((a_{1}, a_{2}, \dots, a_{r}))=P_{(a_{1}, a_{2}, \dots, a_{r}) }$, then
$$
\begin{aligned}
&\qquad \sum_{(a_{1}, \dots, a_{l}) \in \textbf{Ter}(\mathrm{T}_{1} ) }\Big\|\overline{V}(P_{(a_{1}, \dots, a_{l})}P_{(a_{1}, \dots, a_{l-1})} \dots P_{(a_{1})})y \Big\|_{\widetilde{\mathcal{M}}_{U,pV}}^{p}\\
=&\lim_{m\to \infty}
\sum_{(a_{1}, \dots, a_{l}) \in \textbf{Ter}(\mathrm{T}_{1} ) }\Big\|\overline{V}(P_{(a_{1}, \dots, a_{l})}P_{(a_{1}, \dots, a_{l-1})} \dots P_{(a_{1})})\Big(\sum_{i=1}^{N_{m}}C_{i}^{m}\overline{U}_{Q_{i}^{m}, x^{m}_{i}}\Big) \Big\|_{\widetilde{\mathcal{M}}_{U,pV}}^{p}\\
=&\lim_{m\to \infty}
\sum_{(a_{1}, \dots, a_{l}) \in \textbf{Ter}(\mathrm{T}_{1} ) }\Big\|\sum_{i=1}^{N_{m}}C_{i}^{m}\overline{U}_{P_{(a_{1}, \dots, a_{l})}P_{(a_{1}, \dots, a_{l-1})} \dots P_{(a_{1})}Q_{i}^{m},x_{i}^{m}}) \Big\|_{\widetilde{\mathcal{M}}_{U,pV}}^{p} ,\\
\end{aligned}
$$
where sum is taken is over all the terminals in $\textbf{Ter}(\mathrm{T}_{1})$. For a specific terminal $(a_{1}, \dots, a_{l})$, we have
$$
\begin{aligned}
&\Big\|\sum_{i=1}^{N_{m}}C_{i}^{m}\overline{U}_{P_{(a_{1}, \dots, a_{l})} \dots P_{(a_{1})}Q_{i}^{m},x_{i}^{m}} \Big\|_{\widetilde{\mathcal{M}}_{U,pV}}^{p}\\
=&\sup_{\mathrm{T}\in \mathcal{T},\mathscr{P}\in OR(\mathrm{T})}\sum_{(b_{1},\dots, b_{m}) \in \textbf{Ter}(\mathrm{T})}
\Big\|\sum_{i=1}^{N}C_{i}^{m}\overline{U}(P_{(b_{1},\dots, b_{m})}  \dots P_{(b_{1})}
P_{(a_{1}, \dots, a_{l})} \dots P_{(a_{1})}Q_{i}^{m})x_{i}^{m} )\Big\|_{X}^{p}
\end{aligned}
$$
For any finite tree $\mathrm{T}$  with terminal set $\textbf{Ter}(\mathrm{T})=\{ (b_{1}, \dots, b_{m} ) \}$  and an orthogonal representation
$$
\big\{ P_{(b_{1}, \dots, b_{k} )}\big\}_{ (b_{1}, \dots, b_{m} ) \in \textbf{Ter}(\mathrm{T}), 1\leq k \leq m },
$$
we define $\mathrm{T}_{2}$  to  be  a finite tree whose terminal set is  $\{ (a_{1},\dots a_{l}, b_{1}, \dots b_{m})\}$. Then $\mathrm{T}_{1}$ is a subtree of $\mathrm{T}_{2}$. Define mapping  $\mathscr{P}_{2}: \mathrm{T}_{2} \to \mathcal{P}(M) $  as
$$
\mathscr{P}_{2}|_{\mathrm{T}_{1}}=\mathscr{P}_{1}, \text{ and }
\mathscr{P}_{2}((a_{1},\dots a_{l}, b_{1}, \dots b_{k}))=P_{(b_{1}, \dots b_{k})}.
$$
It follows that  $\mathscr{P}_{2}$ is an orthogonal representation of $\mathrm{T}_{2}$.

Thus we have
$$
\begin{aligned}
&\sum_{(a_{1}, \dots, a_{l}) \in \textbf{Ter}(\mathrm{T}_{1} ) }\Big\|\overline{V}(P_{(a_{1}, \dots, a_{l})} \dots P_{(a_{1})})y \Big\|_{\widetilde{\mathcal{M}}_{U,pV}}^{p}\\
&=\lim_{m\to \infty}\sup_{\mathrm{T}_{2}\in \mathcal{T},\mathscr{P}_{2} \in OR( \mathrm{T}_{2} ) } \sum_{(a_{1}, \dots , a_{l}, b_{1},\dots, b_{m}) \in \textbf{Ter}(\mathrm{T}_{2}),}
\Big\|\sum_{i=1}^{N_{m}}C_{i}^{m}\overline{U}(P_{(a_{1}, \dots , a_{l}, b_{1},\dots, b_{m})}, \dots , P_{(a_{1})}Q_{i}^{m})x_{i}^{m} \Big\|_{X}^{p}\\
\end{aligned}
$$
where the  ``sup"  is taken over all finite  tree $\mathrm{T}_{2}$ containing $\mathrm{T}_{1}$ as a subtree and all the  orthogonal representation $ \mathscr{P}_{2}$ of $\mathrm{T}_{2}$ with  $\mathscr{P}_{1}$  as it restriction to subtree $\mathrm{T}_{1}.$

Therefore we obtain
$$
|V_{y}|_{p}^{p}(I)\leq \lim_{m\to \infty} \Big\|\sum_{i=1}^{N_{m}}C_{i}^{m}\overline{U}_{Q_{i}^{m},x_{i}^{m}}\Big\|_{pV}^{p}
=\|y\|_{pV}^{p}
$$
and hence $V$ has contractive  $p$-variation.
\end{proof}


\noindent \textbf{Proof of Theorem \ref{p-dilation} (ii).}  Let  $\{P_{j}\}_{j=1}^{\infty}$ be a countable family of  mutually orthogonal projections in $\mathcal{P}(M)$ with superma $P$ and let $\sum_{i=1}^{N}C_{i}\overline{U}_{Q_{i},x_{i}}\in \mathcal{M}_{U} $. Then we have
$$
\begin{aligned}
&\Big\|V(P)\Big(\sum_{i=1}^{N} C_{i} \overline{U}_{Q_{i}, x_{i}}\Big)-\sum_{j=1}^{M} V(P_{j})\Big(\sum_{i=1}^{N} C_{i} \overline{U}_{Q_{i}, x_{i}}\Big)\Big\|_{pV}\\
=& \Big\|\sum_{i=1}^{N} C_{i}\Big(V(P)\overline{U}_{Q_{i}, x_{i}}-\sum_{j=1}^{M}V(P_{j}) \overline{U}_{Q_{i}, x_{i}}\Big) \Big\|_{pV}\\
\leq &\sum_{i=1}^{N} |C_{i}| \Big\|V(\sum_{j=M+1}^{\infty}P_{j})\overline{U}_{Q_{i}, x_{i}}\Big\|_{pV}
\end{aligned}
$$
We need to prove that
\begin{equation}\label{pcountablyadditive}
\lim_{M\to \infty}\Big\|V(\sum_{j=M+1}^{\infty}P_{j})\overline{U}_{Q_{i}, x_{i}}\Big\|_{pV}=0
\end{equation}

We need the following two facts.

\noindent \textbf{Fact 1.} If $P_{\alpha},P_{\beta}\in \mathcal{P}(M), P_{\alpha}\leq P_{\beta}$,   then  for any $\overline{U}_{Q, x}\in \mathcal{M}_{U}$,
$$
\|V(P_{\alpha})\overline{U}_{Q, x}\|_{pV}\leq \| V(P_{\beta})\overline{U}_{Q, x}\|_{pV},
$$
and thus the sequence $\left\|V(\sum_{j=M+1}^{\infty}P_{j})\overline{U}_{P_{i}, x_{i}}\right\|_{pV}$ is decreasing.

\noindent \textbf{Fact 2.}  If $P_{1},  P_{2} \in \mathcal{P}(M)$ and $P_{1}\perp P_{2}=0$,  then  $|V_{y}|^{p}(P_{1})+ |V_{y}|^{p}(P_{2})\leq |V_{y}|^{p}(P_{1}+P_{2})$.

Indeed, let $\mathrm{T}_{1} ( \subset  \bigcup_{n=1}^{L_{1}} \mathbb{N}^{n} )$  be a  finite tree with the node set $\{(a_{1}, a_{2},\dots, a_{r})\}$ and an orthogonal representation $\{ P_{(a_{1}, a_{2},\dots, a_{r})}\}$, and  $\mathrm{T}_{2} (\subset \bigcup_{n=1}^{M_{1}} \mathbb{N}^{n}) $ be another finite tree with nodes $\{ (b_{1}, b_{2},\dots, b_{k} )\}$ and an orthogonal representation $\{ P_{(b_{1}, b_{2},\dots, b_{k}))}\}$. Define the set
$$
\mathrm{T}\subset \bigcup_{n=1}^{\max\{L_{1}, M_{1}\}+1} \mathbb{N}^{n}
$$
whose elements are in the set  $\{ (0), (1), (0, a_{1}, a_{2},\dots, a_{r}), (1,b_{1}, b_{2},\dots, b_{k} )\}$. Then $\mathrm{T}$ is  a finite tree, and if we define $\mathscr{P}: \mathrm{T} \to \mathcal{P}(M)$ by
$$
\begin{aligned}
\mathscr{P}((0)) &=P_{1}, &  \mathscr{P}((0, a_{1}, a_{2},\dots, a_{r}))&=P_{(a_{1}, a_{2},\dots, a_{r})} \\
\mathscr{P}((1)) & =P_{2},   &  \mathscr{P}((1,b_{1}, b_{2},\dots, b_{k} )) & = P_{(b_{1}, b_{2},\dots, b_{k} )}
\end{aligned}
$$
which is an orthogonal representation of $\mathrm{T}$.  For convenience, we write $\mathscr{P}((c_{1},\dots, c_{i}))=P_{(c_{1},\dots, c_{i})}$ whenever  $(c_{1},\dots c_{i}) \in \mathrm{T}$.

Thus we get
$$
\begin{aligned}
&\sum_{(a_{1}, \dots, a_{l}) \in \textbf{Ter}(\mathrm{T}_{1})}
\left\| \overline{V}(P_{(a_{1}, \dots, a_{l})}\dots P_{(a_{1})} P_{1})y\right\|^{p}+\sum_{(b_{1}, \dots, b_{m}) \in \textbf{Ter}(\mathrm{T}_{2})}\left\| \overline{V}(P_{(b_{1}, \dots, b_{m})}\dots P_{( b_{1} )}P_{2})y \right\|^{p} \\
=&\sum_{(c_{1}, \dots, c_{n}) \in \textbf{Ter}(\mathrm{T})} \left\| \overline{V}( P_{(c_{1}, \dots, c_{n}) } \dots P_{(c_{1})} (P_{1}+P_{2}))y \right\|^{p},
\end{aligned}
$$
Therefore, by the definition of $p$-variation, we get $|V_{y}|^{p}(P_{1})+ |V_{y}|^{p}(P_{2})\leq |V_{y}|^{p}(P_{1}+P_{2})$.

Now we continue our proof.  Assume to the contrary that  (\ref{pcountablyadditive}) does not holds.  Then we can find some $\delta > 0$ and a sequence of $n_{1}\leq m_{1} < n_{2}\leq  m_{2}< n_{3}\leq m_{3} \dots $ such that
\begin{equation}\label{inequ 3}
\Big\|V(\sum_{j=n_{k}}^{m_{k}} P_{j})\overline{U}_{Q_{i},x_{i}}\Big\|_{pV}\geq \delta, ~\forall k \in \mathbb{N}
\end{equation}

Let  $P_{k}^{\prime}=\sum_{j=n_{k}}^{m_{k}} P_{j}$. Then $\{ P_{k}^{\prime}\}_{k=1}^{\infty}$ is a family mutually orthogonal projections. Define $P^{\prime}=\sum_{k=1}^{\infty}P_{k}^{\prime}$. Then for any given positive integer $K$,
$\sum_{k=1}^{K}P_{k}^{\prime}\leq P^{\prime}$. Moreover,  according to our basic facts,  we have
$$
\sum_{k=1}^{K}\left\|V(P_{k}^{\prime})\overline{U}_{Q_{i},x_{i}}\right\|_{pV}^{p} \leq \Big\|V(\sum_{k=1}^{K}P_{k}^{\prime})\overline{U}_{Q_{i},x_{i}}\Big\|_{pV}^{p}\leq  \left\|V(P^{\prime})\overline{U}_{Q_{i},x_{i}}\right\|_{pV}^{p}.
$$

On the other hand, by the boundedness of $V$, we have
$$
\left\|V(P^{\prime})\overline{U}_{Q_{i},x_{i}}\right\|_{pV}^{p}\leq \left\|\overline{U}_{Q_{i},x_{i}}\right\|_{pV}^{p} \leq 4\|U\|_{pV}\|x_{i}\|.
$$
Combining this with  inequality (\ref{inequ 3}) we get
$$
4\|U\|_{pV}\|x_{i}\| \geq K \delta^{p}, ~\forall K \in \mathbb{N}^{+}
$$
which leads to a contradiction. Therefore $V$ is countably additive on the dense subspace $\mathcal{M}_{U}$ of $\widetilde{\mathcal{M}}_{U, pV}$, and so  $V$ is countably additive on $\widetilde{\mathcal{M}}_{U, pV}$ by approximation technique similar to the  Equation (\ref{strongadditive}) in Theorem \ref{countablemeasuredilation}.

\section*{Acknowledgments} The authors would like to express appreciation to  Guixiang Hong, Chi-Keung Ng and Quanhua Xu for many helpful comments and suggestions.

\end{document}